\documentclass[10pt]{article}
\usepackage{mathrsfs}
\usepackage{mathrsfs}
\usepackage{mathrsfs}
\usepackage{amsfonts}
\usepackage{amssymb}
\usepackage{amsmath}
\usepackage{amsthm}
\usepackage{dsfont}
\usepackage[colorlinks,
            linkcolor=black,
            anchorcolor=black,
            citecolor=black
            ]{hyperref}
\theoremstyle{plain}
\newtheorem{thm}{Theorem}[section]

\newtheorem{lem}[thm]{Lemma}
\newtheorem{prop}[thm]{Proposition}

\theoremstyle{definition}
\newtheorem{defn}{Definition}[section]

\theoremstyle{remark}
\newtheorem{rem}{Remark}[section]
\allowdisplaybreaks[4]

\begin{document}
\title{{\Large\bf {Quantum stochastic linear quadratic control theory: Closed-loop solvability}}
\thanks{
This work is supported by National Natural Science Foundation of China(no.12271298 and no.11871308).}}
\author{{\normalsize Penghui Wang,\quad Shan Wang \quad and\quad Shengkai Zhao  \,\,\,\,  } \\
{\normalsize School of Mathematics, Shandong University,} {\normalsize Jinan, 250100, China} }
\date{}
\maketitle
\begin{minipage}{14.1cm}
{\bf Abstract.}
In this paper, we investigate the closed-loop solvability of the quantum stochastic linear quadratic optimal control problem.
We derive the Pontryagin maximum principle for the linear quadratic control problem of infinite-dimensional quantum stochastic systems.
The equivalence between unique closed-loop solvability for quantum stochastic linear quadratic optimal control problems and the well-posedness of the corresponding quantum Riccati equations is established.
Notably, although the quantum Riccati equation is an infinite-dimensional deterministic operator-valued ordinary differential equation, classical methods are not applicable. Inspired by L\"{u} and Zhang's approach [Q. L\"{u} and X. Zhang, Probability Theory and Stochastic Modelling, 101. Springer, Cham, (2021) \& Mem. Amer. Math. Soc. 294 (2024)] to stochastic Riccati equations, we prove the existence and uniqueness of its solutions.
The results  provide a theoretical foundation for the optimal design of quantum control.
\end{minipage}
 \maketitle
\numberwithin{equation}{section}
\newtheorem{theorem}{Theorem}[section]
\newtheorem{lemma}[theorem]{Lemma}
\newtheorem{proposition}[theorem]{Proposition}
\newtheorem{corollary}[theorem]{Corollary}
\newtheorem{remark}[theorem]{Remark}
\indent\indent

\section{Introduction}
\indent\indent
Linear quantum stochastic systems are a class of models used in quantum optics, circuit quantum electrodynamics systems, quantum opto-mechanical systems, and elsewhere \cite{D.P,G-2004,G.C-1985,Z.D}. The mathematical framework for these models is provided by the theory of quantum Wiener processes, and the associated quantum stochastic differential equations. With the rapid development of quantum technology, effectively controlling quantum systems to achieve specific functionalities has become a critical area of research \cite{G.J.N,M.J.P,N.2}. Furthermore, control problems for linear systems often enjoy analytical or computationally tractable solutions. In particular, the linear quadratic control  problem has gained widespread attention due to its effectiveness in optimizing system performance \cite{G.P-2017,G.P-2018,T.N,Z.D}.

In this paper, we study the linear quadratic optimal control problem of quantum stochastic systems. First, we present a brief introduction to noncommutative spaces.
Let $(\Lambda(\mathscr{H}), m, \mathscr{C})$ be a quantum (noncommutative) probability space \cite{B.S.W.1,B.S.W.2,P.X,M.T,W.W.Pontryagin,W.W.The relaxed transposition solution,W.F} on which the anti-symmetric Fock space $\Lambda(\mathscr{H})$ over $\mathscr{H}=L^2(\mathbb R^+)$ is defined. Let $\mathscr{C}$ be the von Neumann algebra generated by $\{\Psi(v): v\in L^2(\mathbb{R}^+)\}$, and let $\{\mathscr{C}_t\}_{t\geq 0}$ denote the increasing family of von Neumann subalgebras of $\mathscr{C}$ generated by  $\left\{\Psi(v): v\in L^2(\mathbb{R}^+)\ \textrm{and}\ \textrm{ess supp}\ v\subseteq [0,t]\right\}$. The Fermion Brownian motion $W(\cdot)$ is given by
\begin{equation}\label{Fermion Brownian motion}
W(t):=\Psi(\chi_{[0,t]})=\mathscr{A}^*(\chi_{[0,t]})+\mathscr{A}(J\chi_{[0,t]}),\quad t\geq0.
\end{equation}
which is self-adjoint and satisfies $W(t)^2=tI$ by the canonical anti-commutation relation (CAR for short) of Fermion fields, where $\mathscr{A}$, $\mathscr{A}^*$ and $J$ are annihilation, creation and complex conjugation operators, respectively. For the Fock vacuum $\Omega\in\Lambda_0(\mathscr{H})\subseteq\Lambda(\mathscr{H})$, define
$m(\cdot):=\langle\Omega, \cdot\Omega\rangle_{\Lambda(\mathscr{H})}$, which is a faithful, normal, central state on $\mathscr{C}$. For any $p\in [1,\infty)$, let $L^p(\mathscr{C})$ denote the completion of $\mathscr{C}$ with the norm $\|f\|_p=m\left(|f|^p\right)^\frac{1}{p}=\left\langle\Omega, |f|^{p}\Omega\right\rangle_{\Lambda(\mathscr{H})}^\frac{1}{p}$ (see \cite{F.K} for details).


Let $\mathcal{X}$ be a Banach space and $T>0$ be a fixed time horizon. Denote by $C([0,T];\mathcal{X})$ the Banach space of all continuous $\mathcal{X}$-valued functions on $[0,T]$. For each $q\in[1,\infty)$,
let $L^q([0,T];\mathcal{X})$ be the Banach space of all  $\mathcal{X}$-valued functions that are $q$th power Lebesgue integrable on $[0,T]$. 
Moreover, $L^\infty([0,T];\mathcal{X})$ is the Banach space of all  $\mathcal{X}$-valued, Lebesgue measurable functions that are essentially bounded on $[0,T]$.   In particular,  let
\begin{gather*}
C_\mathbb{A}([0,T];L^p(\mathscr{C})):=\left\{f\in C([0,T];L^p(\mathscr{C}));\ f(t)\in L^p(\mathscr{C}_t),\  \textrm{a.e.}\ t\in [0,T]\right\},\\
L^q_\mathbb{A}([0,T];L^p(\mathscr{C})):=\left\{f\in L^q([0,T];L^p(\mathscr{C}));\ f(t)\in L^p(\mathscr{C}_t),\  \textrm{a.e.}\ t\in [0,T]\right\}.
\end{gather*}

Let $\mathcal{X}_1$ and $\mathcal{X}_2$ be Banach spaces. Denote by $\mathcal{L}(\mathcal{X}_1;\mathcal{X}_2)$ the Banach space of all bounded
linear operators from $\mathcal{X}_1$ to $\mathcal{X}_2$ with the usual operator norm, and denote $\mathcal{L}(\mathcal{X}_1)$ simply as  $\mathcal{L}(\mathcal{X}_1;\mathcal{X}_1)$. Moreover, define
\begin{gather*}
 C_\mathbb{A}([0,T];\mathcal{L}(\mathcal{X};L^2(\mathscr{C}))):=\left\{F\in C([0,T];\mathcal{L}(\mathcal{X};L^2(\mathscr{C}))); F(t)\xi\in L^2(\mathscr{C}_t), \textrm{a.e.}\ t\in[0,T],    \xi\in \mathcal{X} \right\},\\
         L^q_\mathbb{A}([0,T];\mathcal{L}(\mathcal{X};L^2(\mathscr{C}))):=\left\{F\in L^q([0,T];\mathcal{L}(\mathcal{X};L^2(\mathscr{C}))); F(t)\xi\in L^2(\mathscr{C}_t), \textrm{a.e.}\ t\in[0,T],    \xi\in \mathcal{X} \right\},
 \end{gather*}
 and
\begin{equation*}
 L^\infty_\mathbb{A}([0,T];\mathcal{L}(\mathcal{X};L^2(\mathscr{C}))):=\left\{F\in L^\infty([0,T];\mathcal{L}(\mathcal{X};L^2(\mathscr{C}))); F(t)\xi\in L^2(\mathscr{C}_t), \textrm{a.e.}\ t\in[0,T],    \xi\in \mathcal{X} \right\}.
 \end{equation*}
Let $\mathcal{H}$ be a Hilbert space. Set
\begin{equation*}
  \mathbb{S}(\mathcal{H}):=\left\{F\in\mathcal{L}(\mathcal{H});\ F=F^*\right\},
  \end{equation*}
  and
  \begin{equation*}
 \mathbb{S}_+(\mathcal{H}):=\left\{F\in\mathbb{S}(\mathcal{H});\ \langle F\xi,\xi\rangle_\mathcal{H}\geq 0,\  \xi\in \mathcal{H}\right\}.
  \end{equation*}

Let $U$ be a separable Hilbert space. 
Consider the following controlled linear quantum stochastic system in noncommutative space $L^2(\mathscr{C})$:
\begin{equation}\label{FQSDE-LQ-inital}
\left\{
\begin{aligned}
  dx(t)= &\{A(t)x(t)+B(t)u(t)\}dt+\{C(t)x(t)+D(t)u(t)\}dW(t),\ {\rm{in}}\ [t_0,T],\\
   x(t_0)=& \eta,
\end{aligned}
\right.
\end{equation}
where
\begin{equation}\label{the condition of ABCD}
\left\{
\begin{array}{ll}
A(\cdot)\in L^1_\mathbb{A}([t_0,T];\mathcal{L}(L^2(\mathscr{C}))), &B(\cdot)\in L_\mathbb{A}^2([t_0,T];\mathcal{L}(U; L^2(\mathscr{C}))),\\
 C(\cdot)\in L^2_\mathbb{A}([t_0,T];\mathcal{L}(L^2(\mathscr{C}))), &D(\cdot)\in L_\mathbb{A}^\infty([t_0,T];\mathcal{L}(U; L^2(\mathscr{C}))).
 \end{array}\right.
\end{equation}
In the above, $x(\cdot)$ is the state process, and $u(\cdot)\in \mathcal{U}[t_0,T]:=L^2([t_0,T];U)$ is the control process. Any $u(\cdot)\in \mathcal{U}[t_0,T]$ is called an admissible control. For any initial pair $(t_0,\eta)\in [0,T]\times L^2(\mathscr{C}_{t_0})$ and admissible control $u(\cdot)\in \mathcal{U}[t_0,T]$, it follows from \cite[Theorem 2.1]{B.S.W.2} that the equation \eqref{FQSDE-LQ-inital} admits a unique solution $\bar{x}(\cdot)\equiv x(\cdot\hspace{0.3mm};t_0,\eta,\bar{u}(\cdot))$.

We introduce the following cost functional:
\begin{equation}\label{Quadratic cost functional}
\mathcal{J}(t_0,\eta;u(\cdot))=\frac{1}{2}{\rm Re}\left\{\int_{t_0}^T\{\langle M(t)x(t),x(t)\rangle+\langle R(t)u(t),u(t)\rangle_U\}dt+\langle Gx(T),x(T)\rangle\right\}.
\end{equation}
where
\begin{equation}\label{condition of MGR}
M(\cdot)\in L^1_\mathbb{A}([t_0,T];\mathbb{S}_+( L^2(\mathscr{C}))),\quad R(\cdot)\in L^\infty([t_0,T];\mathbb{S}_+(U)),\quad  G\in\mathbb{S}_+(L^2(\mathscr{C}_T)).
\end{equation}
Here and in what follows, we shall use $\langle\cdot,\cdot\rangle$ for the inner product in $L^2(\mathscr{C})$, where it is conjugate-linear with respect to the first variable and linear with respect to the second variable.

The optimal control problem studied in this paper is as follows.\\
\textbf{Problem (QSLQ).} For any given $(t_0,\eta)\in[0,T]\times L^2(\mathscr{C}_{t_0})$, find a $\bar{u}(\cdot)\in \mathcal{U}[t_0,T]$ such that
\begin{equation}\label{Problem QSLQ}
V(t_0,\eta)=\mathcal{J}(t_0,\eta;\bar{u}(\cdot))=\inf_{u(\cdot)\in \mathcal{U}[t_0,T]}\mathcal{J}(t_0,\eta;u(\cdot)).
\end{equation}
Any $\bar{u}(\cdot)\in \mathcal{U}[t_0,T]$  satisfying \eqref{Problem QSLQ} is called an \textit{optimal control} of \textbf{Problem (QSLQ)} for the initial pair $(t_0,\eta)$, and the corresponding $\bar{x}(\cdot)\equiv x(\cdot\hspace{0.3mm};t_0,\eta,\bar{u}(\cdot))$ is called an \textit{optimal state process}; the pair $(\bar{x}(\cdot),\bar{u}(\cdot))$ is called an \textit{optimal pair}. The function $V(\cdot,\cdot)$ is called the \textit{value function} of {\bf Problem (QSLQ)}.

Similar to  \cite{L.,L.Z,S.Y-Book-2019}, we provide the definitions of optimal feedback operators and the closed-loop solvability  about {\bf Problem (QSLQ)}.
\begin{defn}\label{optimal feedback operator}
A stochastic process $\Theta(\cdot)\in L^2([t_0,T]; \mathcal{L}(L^2(\mathscr{C});U))$ is called an optimal feedback operator for \textbf{Problem (QSLQ)} on $[t_0,T]$ if
\begin{equation}\label{Comparison of cost functional}
\mathcal{J}(t_0,\eta;\Theta(\cdot)\bar{x}(\cdot))\leq\mathcal{J}(t_0,\eta;u(\cdot)),\quad  \eta\in L^2(\mathscr{C}_{t_0}),\ u(\cdot)\in \mathcal{U}[t_0,T],
\end{equation}
where $\bar{x}(\cdot)=x(\cdot\, ;t_0,\eta,\Theta(\cdot)\bar{x}(\cdot))$ is the solution to \eqref{FQSDE-LQ-inital} with $\bar{u}(\cdot)=\Theta(\cdot)\bar{x}(\cdot)$.
\end{defn}
\begin{defn}
\textbf{Problem (QSLQ)} is said to be (uniquely) closed-loop solvable on $[t_0,T]$ if an optimal feedback operator (uniquely) exists on $[t_0,T]$.
\end{defn}

Closed-loop feedback control is an effective method for addressing linear quadratic control problems. In this framework,  \textbf{Problem (QSLQ)} typically involves minimizing a quadratic cost functional that considers both the state process and the control process, thereby ensuring the realization of an optimal control strategy.
By formulating an optimization problem, the Riccati equation can be derived from the perspectives of the calculus of variations or dynamic programming.
This provides the necessary theoretical foundation for describing optimal control, extensively studied in classical optimal linear quadratic control theory, as noted in \cite{A.M.Z, L., L.Z-2024, K.T-2002, K.T-2003, S.L.Y, S.Y,S.Y-Book-2019,W.M.W}.
Analogous to classical closed-loop feedback control theory, the following quantum Riccati equation is the main technique for studying \textbf{Problem (QSLQ)}:
\begin{equation}\label{quantum Riccati equation}
\left\{
\begin{aligned}
&\frac{dP}{dt}=\left\{-PA-A^*P-C^*PC-M+L^*K^{-1}L\right\},\ {\rm{in}}\ [t_0,T],\\
&P(T)=G,
\end{aligned}
\right.
\end{equation}
where
\begin{equation}\label{the definition of K and L}
L:=B^*P+D^*PC,\quad K:=R+D^*PD,
\end{equation}
and $K^{-1}$ is inverse of $K$. 

Next, we introduce the following quantum stochastic differential equations in $L^2(\mathscr{C})$:
\begin{equation}\label{to transposition solution QSDE-1}
\left\{
\begin{aligned}
 dz_1(s) &=\left\{A(s)z_1(s)+\mu_1(s)\right\}ds+\{C(s)z_1(s)+\nu_1(s)\}dW(s),\ {\rm in}\  [t,T], \\
 z_1(t)& =\xi_1,\\
\end{aligned}
\right.
\end{equation}
and
\begin{equation}\label{to transposition solution QSDE-2}
\left\{
\begin{aligned}
 dz_2(s) &=\left\{A(s)z_2(s)+\mu_2(s)\right\}ds+\{C(s)z_2(s)+\nu_2(s)\}dW(s),\ {\rm in}\ [t,T], \\
 z_2(t)& =\xi_2,
\end{aligned}
\right.
\end{equation}
where $\xi_1,\xi_2\in L^2(\mathscr{C}_{t})$, $\mu_1(\cdot),\mu_2(\cdot),\nu_1(\cdot),\nu_2(\cdot)\in L^2_\mathbb{A}([t,T];L^2(\mathscr{C}))$. It can be shown that \eqref{to transposition solution QSDE-1} (resp.\eqref{to transposition solution QSDE-2}) admits a unique solution $z_1(\cdot)\in C_\mathbb{A}([t_0,T];L^2(\mathscr{C}))$ (resp.$z_2(\cdot)\in C_\mathbb{A}([t_0,T];L^2(\mathscr{C}))$).
\begin{defn}\label{the def of weak solution to Riccati}
We call $P(\cdot)\in C_\mathbb{A}([t_0,T];\mathbb{S}(L^2(\mathscr{C})))$ a weak solution to \eqref{quantum Riccati equation} if the following conditions hold:
\begin{description}
  \item[(i)] $K(t)(\equiv R(t)+D(t)^*P(t)D(t))>0$ and its left inverse $K(t)^{-1}$ is a densely defined closed operator for a.e. $t\in[t_0,T]$;
  \item[(ii)] For any $t\in[t_0,T]$,  $\xi_1,\xi_2\in L^2(\mathscr{C}_{t})$, $\mu_1(\cdot),\mu_2(\cdot),\nu_1(\cdot),\nu_2(\cdot)\in L^2_\mathbb{A}([t,T];L^2(\mathscr{C}))$, it holds that
\begin{align*}
&\langle Gz_1(T), z_2(T)\rangle+\int_{t}^{T}\langle M(s) z_1(s), z_2(s)\rangle ds-\int_{t}^{T}\langle L(s)^*K(s)^{-1}L(s) z_1(s), z_2(s)\rangle ds\\
&=\langle P(t)z_1(t), z_2(t)\rangle+\int_{t}^{T}\langle P(s) z_1(s), \mu_2(s)\rangle ds+\int_{t}^{T}\langle P(s)\mu_1(s),z_2(s)\rangle ds\\
&\hspace{6mm}+\int_{t}^{T}\langle P(s)\{C(s)z_1(s)+\nu_1(s)\},\nu_2(s)\rangle ds +\int_{t}^{T}\langle P(s) \nu_1(s),C(t)z_2(s)\rangle ds,
\end{align*}
where $z_1(\cdot)$ and $z_2(\cdot)$ solve \eqref{to transposition solution QSDE-1} and \eqref{to transposition solution QSDE-2}, respectively.
\end{description}
\end{defn}

 We present the main result, which reveals the relationship between the closed-loop solvability of \textbf{Problem (QSLQ)} and the existence of solutions to the quantum Riccati equation \eqref{quantum Riccati equation}.
\begin{thm}\label{the regular solution and the feedback operator}
\textbf{Problem (QSLQ)} is  uniquely closed-loop solvable if and only if the quantum Riccati equation \eqref{quantum Riccati equation} admits a uniqueness weak solution $P(\cdot)$ in $C_\mathbb{A}([t_0,T];\mathbb{S}(L^2(\mathscr{C})))$. In this case, the optimal feedback operator $\Theta(\cdot)$ is given by
\begin{equation}\label{thm-feedback operator}
 \Theta(\cdot)=-K(\cdot)^{-1} L(\cdot),
\end{equation}
and the value function is
\begin{equation}\label{the relation of value function and P}
V(t_0,\eta)=\frac{1}{2}{\rm{Re}}\langle P(t_0)\eta,\eta\rangle.
\end{equation}
\end{thm}

Belavkin \cite{B} was the first to develop the mathematical theory of feedback control in quantum systems. Subsequent work by Edwards, Nurdin, James, and others further explored quantum filtering dynamics and optimal feedback control. They applied dynamic programming principles to investigate the optimal control problem for finite-dimensional quantum systems and derived the corresponding finite-dimensional Riccati equations \cite{D.J,E.B,J.N.P,N.J.P,W.J}.

To obtain the above result, we must overcome the following difficulties:
\begin{itemize}
\item The quantum Riccati equations currently being studied are typically finite-dimensional \cite{E.B,J.N.P}, i.e., matrix-valued, with relatively little research on infinite-dimensional quantum Riccati equations.
\item Unlike standard one-dimensional Brownian motions \cite{L.,S.L.Y,S.Y}, the Fermion Brownian motion $W(\cdot)$ does not commute with the diffusion term.

\item 
In this case, although the quantum Riccati equation \eqref{quantum Riccati equation} is deterministic, the images of its solutions acting on any element of its domain form  an adapted process.
Therefore, the classical methods \cite{L.,S.L.Y} for solving the deterministic Riccati equation are not applicable.
 \item In the study of classical infinite-dimensional Riccati equations \cite{L.,L.Z,L.Z-2024}, since forward (backward) stochastic differential equations possess sample paths, the resulting equations obtained through orthogonal projection onto finite-dimensional subspaces preserve their original adaptedness. In contrast, quantum stochastic differential equations lack sample path structures, thus rendering this methodology inapplicable to solving quantum Riccati equations.
\end{itemize}

To achieve this, we first truncate the initial (terminal) values of the quantum stochastic differential equations, which allows us to obtain approximate solutions for these equations. Next, we derive the Pontryagin maximum principle for \textbf{Problem (QSLQ)} concerning infinite-dimensional quantum stochastic systems \eqref{FQSDE-LQ-inital}. This derivation is based on the canonical anticommutation relations (CAR) of Fermion fields and classical variational methods. Our result builds upon previous contributions in the quantum stochastic theory of Fermion fields by scholars such as Gardiner \cite{G-2004,G.C-1985}, Gough \cite{G.G.J.N}, and Hudson and Parthasarathy \cite{B.H.J,H.L.,P.H,P.book}. Notably, the diffusion terms in the operator-valued quantum stochastic differential equations must incorporate the action of the parity operator $\Upsilon$ \cite{B.S.W.1,B.S.W.2,P.X}, which is both self-adjoint and unitary. Finally, we prove that the unique existence of optimal feedback controls for the linear quadratic optimal control problem
is equivalent to the well-posedness of solutions to the quantum Riccati equation \eqref{quantum Riccati equation}.
This result provide new approaches and insights for the control and optimization of \textit{infinite-dimensional quantum stochastic systems}.

The organization of this paper is as follows.
Section \ref{pre} presents the relevant results concerning the solutions of forward and backward quantum stochastic differential equations. In section \ref{QSLQ problem}, we obtain the Pontryagin-type maximum principle for \textbf{Problem (QSLQ)}.
Section \ref{relationship} is devoted to proving Theorem \ref{the regular solution and the feedback operator}. Sections \ref{Sufficiency} and \ref{Necessity} prove the sufficiency and necessity of
 Theorem \ref{the regular solution and the feedback operator}, respectively.
In the appendix, we give the proof for preliminary result in section \ref{pre}.

\section{Preliminaries}\label{pre}
\indent\indent
In this section, we provide some preliminaries which would be useful in the sequel. 
For any $t_0\in [0,T]$,  consider the  quantum stochastic differential equation (QSDE for short):
\begin{equation}\label{QSDE-solvability}
\left\{
\begin{aligned}
&dx(t)=\{A(t)x(t)+f(t)\}dt+\{C(t)x(t)+g(t)\}dW(t),\quad \textrm{in}\ [t_0, T],\\
&x(t_0)=\eta,
\end{aligned}\right.
\end{equation}
where $\eta\in L^2(\mathscr{C}_{t_0})$, $A\in L^1_\mathbb{A}([t_0,T];\mathcal{L}(L^2(\mathscr{C})))$, $C\in L^2_\mathbb{A}([t_0,T];\mathcal{L}(L^2(\mathscr{C})))$, $f\in L^1_\mathbb{A}([t_0,T];L^2(\mathscr{C}))$ and $g\in L^2_\mathbb{A}([t_0,T];L^2(\mathscr{C}))$. 
\begin{lem}\cite[Theorem 2.1]{B.S.W.2}\label{the property of the solution of FQSDE}
The equation \eqref{QSDE-solvability} has a unique  solution $x(\cdot)\in C_\mathbb{A}([t_0,T];L^2(\mathscr{C}))$,
and
\begin{equation}\label{the estimate of the solution of FQSDE}
\sup_{t\in[t_0,T]}\|x(t)\|_2\leq \mathcal{C}\left(\|\eta\|_2+\|f\|_{L^1_\mathbb{A}([t_0,T];L^2(\mathscr{C}))}+\|g\|_{L^2_\mathbb{A}([t_0,T];L^2(\mathscr{C}))}\right).
\end{equation}
\end{lem}
Consider the following backward quantum stochastic differential equation (BQSDE for short):
\begin{equation}\label{BQSDE-solvability}
\left\{
\begin{aligned}
&dy(t)=\{A(t)^*y(t)+C(t)^*Y(t)+h(t)\}dt+Y(t)dW(t),\quad \textrm{in}\ [t_0, T],\\
&y(T)=\xi,
\end{aligned}\right.
\end{equation}
where $\xi\in L^2(\mathscr{C}_T)$ and $h\in L^1_{\mathbb{A}}([t_0,T];L^2(\mathscr{C}))$.
\begin{lem}\cite[Theorem 4.1]{W.W.Pontryagin}\label{the property of the solution of BQSDE}
For any $\xi\in L^2(\mathscr{C}_T)$, the equation \eqref{BQSDE-solvability} admits a unique solution $(y(\cdot),Y(\cdot))\in C_\mathbb{A}([t_0,T];L^2(\mathscr{C}))\times L^2_\mathbb{A}([t_0,T];L^2(\mathscr{C}))$, and it holds that
\begin{equation}\label{the estimate of the solution of BQSDE}
\|(y(\cdot),Y(\cdot))\|^2_{C_\mathbb{A}([t_0,T];L^2(\mathscr{C}))\times L^2_\mathbb{A}([0,T];L^2(\mathscr{C}))}\leq \mathcal{C}\left(\|\xi\|^2_2+\|h\|^2_{L^1_\mathbb{A}([t_0,T];L^2(\mathscr{C}))}\right).
\end{equation}
\end{lem}
By  Lemma \ref{the property of the solution of FQSDE} and Lemma \ref{the property of the solution of BQSDE}, we can obtain the following result.
\begin{lem}\label{the result on FBQSDE}
Let $\Theta(\cdot)$ be an optimal feedback operator of \textbf{Problem (QSLQ)}. The triple $(x(\cdot),$  $y(\cdot),Y(\cdot))$ is the unique solution to the following forward-backward QSDE\textnormal{:} 
\begin{equation}\label{FBQSDE}
\left\{
\begin{array}{ll}
dx(t)=\{A(t)+B(t)\Theta(t)\}x(t)dt+\{C(t)+D(t)\Theta(t)\}x(t)dW(t),\ &{\rm in}\ [t_0,T],\\
dy(t)=-\left\{A(t)^*y(t)+C(t)^*Y(t)-M(t)x(t)\right\}dt+Y(t)dW(t),\ &{\rm in}\ [t_0, T],\\
x(t_0)=\varsigma,\quad y(T)=-Gx(T).&
\end{array}
\right.
\end{equation}
Moreover,
\begin{equation}\label{estimate of FBQSDE}
\|(x(\cdot),y(\cdot),Y(\cdot))\|_{C_\mathbb{A}([t_0,T];L^2(\mathscr{C}))\times C_\mathbb{A}([t_0,T];L^2(\mathscr{C}))\times L^2_\mathbb{A}([t_0,T];L^2(\mathscr{C}))}\leq \mathcal{C}\|\varsigma\|_2.
\end{equation}
\end{lem}
\begin{lem}\label{dense}
Let $\Theta(\cdot)$ be an optimal feedback operator of \textbf{Problem (QSLQ)}. Then, for any $t\in [t_0,T]$, the following set
\begin{equation*}
\begin{aligned}
 \mathfrak{S}:= \Big\{x(t;\varsigma);\  x(t;\varsigma)=\varsigma &+\int_{t_0}^{t}\{A(s)+B(s)\Theta(s)\}x(s;\varsigma)ds\\
 &+\int_{t_0}^{t}\{C(s)+D(s)\Theta(s)\}x(s;\varsigma)dW(s),\ \varsigma\in L^2(\mathscr{C}_{t_0})\Big\}
  \end{aligned}
\end{equation*}
is dense in $L^2(\mathscr{C}_t)$.
\end{lem}
\begin{proof}
Let us prove this conclusion by contradiction. If this is not the case, then, for given $t\in [t_0,T]$, we can find a nonzero $\rho\in L^2(\mathscr{C}_t)$ such that
\begin{equation}\label{assume}
\langle \rho, x(t;\varsigma)\rangle=0,\quad x(t;\varsigma)\in\mathfrak{S}.
\end{equation}
For the above $\rho$, we consider the following BQSDE:
\begin{equation}\label{BQSDE-dense}
\left\{
\begin{aligned}
&d\alpha(s)=-\left\{(A(s)+B(s)\Theta(s))^*\alpha(s)+(C(s)+D(s)\Theta(s))^*\beta(s)\right\}ds\\
&\hspace{1.4cm}+\beta(s)dW(s),\quad\quad {\rm in}\  [t_0,t],\\
&\alpha(t)=\rho.
\end{aligned}
\right.
\end{equation}
It is clear that \eqref{BQSDE-dense} admits a unique solution  $(\alpha(\cdot),\beta(\cdot))\in C_\mathbb{A}([t_0,t];L^2(\mathscr{C}))\times L^2_\mathbb{A}([t_0,t];L^2(\mathscr{C}))$.
Applying Fermion It\^{o}'s formula \cite[Theorem 5.2]{A.H-2}, we can obtain that
\begin{align*}
&  \langle\alpha(t), x(t;\varsigma)\rangle-  \langle\alpha(t_0), \varsigma\rangle \\
  &=-\int_{t_0}^{t} \langle(A(s)+B(s)\Theta(s))^*\alpha(s)+(C(s)+D(s)\Theta(s))^*\beta(s), x(s;\varsigma)\rangle ds\\
  &\hspace{5mm} +\int_{t_0}^{t} \langle \alpha(s),\{A(s)+B(s)\Theta(s)\}x(s;\varsigma)\rangle ds+\int_{t_0}^{t}\langle\beta(s), \{C(s)+D(s)\Theta(s)\}x(s;\varsigma) \rangle ds.
\end{align*}
This, together with \eqref{assume}, shows that
\begin{equation*}
 \langle\alpha(t_0), \varsigma\rangle =0.
\end{equation*}
According to the arbitrariness of $\varsigma$, $\alpha(t_0)=0.$

Similarly, for any $s\in [t_0,t)$, we can obtain that $\alpha(s)=0.$ Moreover, due to the continuity of $\alpha(\cdot)$, this contradicts $\alpha(t)=\rho\neq0$.
\end{proof}

For the sake of subsequent research, consider the following QSDE:
\begin{equation}\label{QSDE-widetilde-x}
 \left\{
 \begin{aligned}
   d\widetilde{x}(t)=&\left\{-A(t)-B(t)\Theta(t)+(C(t)+D(t)\Theta(t))^2\right\}^*\widetilde{x}(t)dt\\
  & \hspace{0.5mm}-\{C(t)+D(t)\Theta(t)\}^*\widetilde{x}(t)dW(t),\  {\rm{in}}\  [t_0,T],\\
   \widetilde{x}(t_0)=&\gamma,
 \end{aligned}
 \right.
\end{equation}
where $\gamma\in L^2(\mathscr{C}_{t_0})$. Clearly, the equation \eqref{QSDE-widetilde-x} admits a unique solution $\widetilde{x}(\cdot)\in C_\mathbb{A}([t_0,T];L^2(\mathscr{C}))$.
Next, we consider the approximated solutions of \eqref{FBQSDE} and \eqref{QSDE-widetilde-x}, respectively.
Let $\{e_j\}_{j=1}^\infty$  be an orthonormal basic of $L^2(\mathscr{C})$, and let $\Gamma_n$ denote the orthonormal projection from $L^2(\mathscr{C})$  onto its subspace $\textrm{span} \{e_j;\hspace{0.5mm} 1 \leq  j \leq n\}$.

For any $\varsigma, \gamma\in L^2(\mathscr{C}_{t_0})$, consider the following approximated QSDEs:
\begin{equation}\label{FBQSDE-xn-yn}
 \left\{
\begin{aligned}
 &dx_n(t)=\{A(t)+B(t)\Theta(t)\}x_n(t)dt+\{C(t)+D(t)\Theta(t)\}x_n(t)dW(t),\ & {\rm{in}}\ [t_0,T],\\
&dy_n(t)=-\{A(t)^* y_n(t)+C(t)^* Y_n(t)-M(t)x_n(t)\}dt+Y_n(t)dW(t),\ &{\rm{in}}\ [t_0, T],\\
&x_n(t_0)=\Gamma_n\varsigma,\quad y_n(T)=-Gx_n(T),&
\end{aligned}
 \right.
\end{equation}
and
\begin{equation}\label{QSDE-widetilde-xn}
 \left\{
 \begin{aligned}
   d\widetilde{x}_n(t)=&\hspace{0.5mm}\left\{-A(t)-B(t)\Theta(t)+(C(t)+D(t)\Theta(t))^2\right\}^*\widetilde{x}_n(t)dt\\
  &\hspace{0.5mm} -\left\{C(t)+D(t)\Theta(t)\right\}^*\widetilde{x}_n(t)dW(t),\ {\rm{in}}\  [t_0,T],\\
  \widetilde{x}_n(t_0)= &\hspace{0.5mm} \Gamma_n\gamma.
 \end{aligned}
 \right.
\end{equation}
By Lemma \ref{the result on FBQSDE}, the equation \eqref{FBQSDE-xn-yn} has a unique solution $(x_n(\cdot),y_n(\cdot),Y_n(\cdot))\in C_\mathbb{A}([t_0,T];L^2(\mathscr{C}))\times C_\mathbb{A}([t_0,T];L^2(\mathscr{C}))\times L^2_\mathbb{A}([t_0,T];L^2(\mathscr{C}))$, and  it holds that
\begin{equation}\label{estimate of FBQSDE-xnynYn}
\|(x_n(\cdot),y_n(\cdot),Y_n(\cdot))\|_{C_\mathbb{A}([t_0,T];L^2(\mathscr{C}))\times C_\mathbb{A}([t_0,T];L^2(\mathscr{C}))\times L^2_\mathbb{A}([t_0,T];L^2(\mathscr{C}))}\leq \mathcal{C}\|\varsigma\|_2.
\end{equation}
By Lemma \ref{the property of the solution of FQSDE}, the equation \eqref{QSDE-widetilde-xn} admits a unique solution $\widetilde{x}_n(\cdot) \in C_\mathbb{A}([t_0,T];L^2(\mathscr{C}))$, and
\begin{equation}\label{estimate of FQSDE-xn}
\|\widetilde{x}_n(\cdot)\|_{C_\mathbb{A}([t_0,T];L^2(\mathscr{C}))}\leq \mathcal{C}\|\varsigma\|_2.
\end{equation}
Then, we have the following result, and the proof is provided in the Appendix.

\begin{lem}\label{approximate of FBQSDEs}
Under the given conditions above, for any $\varsigma,\gamma\in L^2(\mathscr{C})$, let $(x(\cdot),y(\cdot),Y(\cdot)) ($resp. $\widetilde{x}(\cdot))$ satisfy \eqref{FBQSDE} $($resp. \eqref{QSDE-widetilde-x}$)$,  it holds that
\begin{equation}\label{finite approximation}
\left\{
  \begin{matrix}
&\lim\limits_{n\to\infty}x_n(\cdot)=x(\cdot),&\ {\rm{ in}}\ C_\mathbb{A}([t_0,T]; L^2(\mathscr{C})),\\
&\lim\limits_{n\to\infty}y_n(\cdot)=y(\cdot),&\ {\rm{in}}\ C_\mathbb{A}([t_0,T]; L^2(\mathscr{C})),\\
&\hspace{0.7mm}\lim\limits_{n\to\infty}Y_n(\cdot)\hspace{0.1mm}=Y(\cdot),&\ {\rm{in}}\ L^2_\mathbb{A}([t_0,T]; L^2(\mathscr{C})),\\
&\lim\limits_{n\to\infty}\widetilde{x}_n(\cdot)=\widetilde{x}(\cdot),&\ {\rm{ in}}\ C_\mathbb{A}([t_0,T]; L^2(\mathscr{C})).
\end{matrix}
\right.
\end{equation}
\end{lem}
\section{Pontryagin maximum principle for Problem (QSLQ)}\label{QSLQ problem}
\indent\indent
This section focuses on deriving the Pontryagin maximum principle for \textbf{Problem (QSLQ)}.
\begin{thm}\label{Pontryagin-type maximum principle for SLQ problem}
Let \textbf{Problem (QSLQ)} be solvable at $\eta\in L^2(\mathscr{C}_{t_0})$ with $(\bar{x}(\cdot),\bar{u}(\cdot))$ being an optimal pair of quantum stochastic control system \eqref{FQSDE-LQ-inital}. Then, for the solution $(y(\cdot), Y(\cdot))$ to 
 \begin{equation}\label{SLQ-B-y}
\left\{
\begin{aligned}
&dy(t)=-\left\{A(t)^*y(t)+C(t)^*Y(t)-M(t)\bar{x}(t)\right\}dt+Y(t)dW(t),\ {\rm{in}}\ [t_0,T],\\
&y(T)=-G\bar{x}(T),
\end{aligned}
\right.
\end{equation}
and the realxed transpsotion solution $(\phi(\cdot), \Phi^{(\cdot)},\widehat{\Phi}^{(\cdot)})$, introduced by \cite[Definition 1.1]{W.W.The relaxed transposition solution}, to the following adjoint equation
\begin{equation}\label{SLQ-B-p}
\left\{
\begin{aligned}
&d\phi(t)=-\left\{A^*\phi+\phi A+C^*\phi C+\Phi \Upsilon C + C^*\Phi\Upsilon-M\right\}dt+\Phi(t) dW(t),\ {\rm{in}}\ [t_0,T],\\
&\phi(T)=-G,
\end{aligned}
\right.
\end{equation}
it holds that
\begin{equation}\label{Firstly condition}
  R(t)\bar{u}(t)-B(t)^*y(t)-D(t)^*Y(t)=0,\ {\rm{a.e.}}\ [t_0,T],
\end{equation}
and
\begin{equation}\label{Secondly condition}
{\rm{ Re}}\langle \{R(t)-D(t)^*\phi(t)D(t)\}u,u\rangle\geq0,\  {\rm{a.e.}}\ [t_0,T],\   u\in \mathcal{U}[t_0,T].
\end{equation}
\end{thm}
\begin{proof}
We divide the proof into two steps.

{\bf{Step 1.}} In this step, we shall prove \eqref{Firstly condition} by the convex perturbation technique.
For the optimal pair $(\bar{x}(\cdot),\bar{u}(\cdot))$ and a control variable $u(\cdot)\in L^2([t_0,T]; U)$, 
we have that, for $\varepsilon\in[0,1]$,
\begin{equation*}
u^\varepsilon(\cdot)=\bar{u}(\cdot)+\varepsilon(u(\cdot)-\bar{u}(\cdot))=(1-\varepsilon)\bar{u}(\cdot)+\varepsilon u(\cdot)\in L^2([t_0,T]; U).
\end{equation*}
Let $x^\varepsilon(\cdot)$ be the solution of \eqref{FQSDE-LQ-inital} corresponding to the control $u^\varepsilon(\cdot)$, that is,
\begin{equation*}
\left\{
\begin{aligned}
dx^\varepsilon(t)=&\{A(t)x^\varepsilon(t)+B(t)u^\varepsilon(t)\}dt+\{C(t)x^\varepsilon(t)+D(t)u^\varepsilon(t)\}dW(t),\  \textrm{in}\ [t_0,T],\\
x^\varepsilon(t_0)=&\eta.
\end{aligned}
\right.
\end{equation*}
For any $t\in [0,T]$,
let
\begin{equation}\label{definition of x^e}
x_1^\varepsilon(t):=\frac{1}{\varepsilon}(x^\varepsilon(t)-\bar{x}(t)),\quad \delta u(t):=u(t)-\bar{u}(t).
\end{equation}
It can be easily seen that $x_1^\varepsilon(\cdot)$ is the solution to the following QSDE:
\begin{equation}\label{LQ-FSDE}
\left\{
\begin{aligned}
dx_1^\varepsilon(t)=&\left\{A(t)x_1^\varepsilon(t)+B(t)\delta u(t)\right\}dt+\left\{C(t)x_1^\varepsilon(t)+D(t)\delta u(t)\right\}dW(t),\ \textrm{in}\ [t_0,T],\\
x_1^\varepsilon(t_0)=&0.
\end{aligned}
\right.
\end{equation}
Since $(\bar{x}(\cdot),\bar{u}(\cdot))$ is an optimal pair of \textbf{Problem (QSLQ)}, and $M(\cdot),R(\cdot),G$ are positive, we obtain that
\begin{equation}\label{LQ-1}
\begin{aligned}
  0&\leq \lim_{\varepsilon\to0}\frac{\mathcal{J}(t_0,\eta;u^\varepsilon(\cdot))-\mathcal{J}(t_0,\eta;\bar{u}(\cdot)) }{\varepsilon}\\
  &=\textrm{Re}\int_{t_0}^T\left\{\langle  M(t)\bar{x}(t), x_1^\varepsilon(t)\rangle+\langle R(t)\bar{u}(t),\delta u(t)\rangle_U\right\}dt+\textrm{Re}\langle G\bar{x}(T), x_1^\varepsilon(T)\rangle.
  \end{aligned}
\end{equation}
Applying Fermion It\^{o}'s formula \cite[Theorem 5.2]{A.H-2} to $\langle y(t),x_1^\varepsilon(t)\rangle$, one has
\begin{equation}\label{ito-x1-y}
 \begin{aligned}
 &-\langle G\bar{x}(T) ,x_1^\varepsilon(T)\rangle\\
& =\int_{t_0}^T\langle  y(t),A(t)x_1^\varepsilon(t)+B(t)\delta u(t)\rangle dt+\int_{t_0}^T\langle Y(t), C(t)x_1^\varepsilon(t)+D(t)\delta u(t)\rangle dt \\
 &\quad-\int_{t_0}^T\left\{\langle A(t)^*y(t), x_1^\varepsilon(t)\rangle+\langle Y(t), C(t)x_1^\varepsilon(t)\rangle-\langle M(t)\bar{x}(t), x_1^\varepsilon(t)\rangle\right\} dt\\
& =\int_{t_0}^T\{\langle y(t), B(t)\delta u(t)\rangle +\langle Y(t), D(t)\delta u(t)\rangle+\langle M(t)\bar{x}(t), x_1^\varepsilon(t)\rangle\} dt.
 \end{aligned}
\end{equation}
By substituting \eqref{ito-x1-y} into \eqref{LQ-1}, we obtain that, for any $u\in L^2([t_0,T];U)$,
\begin{equation*}
 \textrm{Re}\int_{t_0}^T\langle R(t)\bar{u}(t)-B(t)^*y(t)-D(t)^*Y(t), \delta u(t)\rangle_Udt\geq0.
\end{equation*}
Then, 
\begin{equation*}
   R(t)\bar{u}(t)-B(t)^*y(t)-D(t)^*Y(t)=0,\ \textrm{a.e.}\ t\in [t_0,T].
\end{equation*}
{\textbf{Step 2.}} In this step, we shall prove \eqref{Secondly condition} by the spike variation method.
For each $\varepsilon>0$ and $\tau\in [t_0,T-\varepsilon)$, let $E_\varepsilon:=[\tau,\tau+\varepsilon]$. For any $u\in L^2([t_0,T];U)$, put
\begin{equation*}
u^\varepsilon(t):=\left\{
\begin{aligned}
\bar{u}(t),\quad    &t\in[t_0,T]\setminus E_\varepsilon,\\
u(t),\quad    &t\in E_\varepsilon.
\end{aligned}
\right.
\end{equation*}
Let $x^\varepsilon(\cdot)$ be the solution to \eqref{FQSDE-LQ-inital} with the corresponding to the control $u^\varepsilon(\cdot)$.
 Consider the following two QSDEs:
\begin{equation}\label{FQSDE-LQ-x2}
\left\{
\begin{aligned}
dx_2^\varepsilon(t)=&A(t)x_2^\varepsilon(t)dt+\left\{C(t)x_2^\varepsilon(t)+\chi_{E_\varepsilon}(t)D(t)\delta u(t)\right\}dW(t),\  \textrm{in}\ [t_0,T],\\
x_2^\varepsilon(t_0)=&0,
\end{aligned}
\right.
\end{equation}
and
\begin{equation}\label{FQSDE-LQ-x3}
\left\{
\begin{aligned}
dx_3^\varepsilon(t)&=\left\{A(t)x_3^\varepsilon(t)+\chi_{E_\varepsilon}(t)B(t)\delta u(t)\right\}dt+C(t)x_3^\varepsilon(t)dW(t),\  \textrm{in}\ [t_0,T],\\
x_3^\varepsilon(t_0)&=0.
\end{aligned}
\right.
\end{equation}
It is clear that $x^\varepsilon-\bar{x}=x_2^\varepsilon+x_3^\varepsilon.$ By \cite[Theorem 3.2]{W.W.Pontryagin}, it holds that
\begin{equation}\label{estimate-x2-x3}
\sup_{t\in[t_0,T]}\|x_2^\varepsilon(t)\|_2\leq \mathcal{C} \sqrt{\varepsilon},\
\sup_{t\in[t_0,T]}\|x_3^\varepsilon(t)\|_2\leq \mathcal{C} \varepsilon.
\end{equation}
Hence, we can get
\begin{equation}\label{cost functional with estimate}
\begin{aligned}
 &\mathcal{J}(t_0,\eta;u^\varepsilon(\cdot))-\mathcal{J}(t_0,\eta;\bar{u}(\cdot)) \\
 & =\textrm{Re}\int_{t_0}^T\left\{\left\langle M(t)\bar{x}(t), x^\varepsilon_2(t)+x_3^\varepsilon(t)\right\rangle+\frac{1}{2}\langle M(t) x^\varepsilon_2(t), x_2^\varepsilon(t)\rangle\right.\\
  &\quad\quad\indent\indent+\chi_{E_\varepsilon}(t)\left.\left(\langle R(t)\bar{u}(t),\delta u(t)\rangle_U+\frac{1}{2}\langle R(t)\delta u(t),\delta u(t)\rangle_U\right) \right\}dt\\
  &\hspace{4.5mm}+\textrm{Re}\langle G\bar{x}(T), x^\varepsilon_2(T)+x_3^\varepsilon(T)\rangle+\frac{1}{2}\textrm{Re}\langle G x^\varepsilon_2(T), x_2^\varepsilon(T)\rangle+\textit{\textbf{o}}(\varepsilon).
\end{aligned}
\end{equation}
Applying Fermion It\^{o}'s formula  \cite[Theorem 5.2]{A.H-2}  to $\langle y(t),x_2^\varepsilon(t)\rangle$ and  $\langle y(t),x_3^\varepsilon(t)\rangle$ again, we have that 
\begin{equation}\label{ito-y-x2}
 -\langle G\bar{x}(T), x_2^\varepsilon(T) \rangle
=\int_{t_0}^T\{\chi_{E_\varepsilon}(t)\langle Y(t), D(t)\delta u(t)\rangle+\langle M(t)\bar{x}(t), x_2^\varepsilon(t)\rangle\} dt,
\end{equation}
and
\begin{equation}\label{ito-y-x3}
 -\langle G\bar{x}(T),x_3^\varepsilon(T)\rangle=  \int_{t_0}^T\{\chi_{E_\varepsilon}(t)\langle y(t),B(t)\delta u(t)\rangle+\langle M(t)\bar{x}(t),x_3^\varepsilon(t)\rangle\} dt.
\end{equation}
From \eqref{ito-y-x2} and \eqref{ito-y-x3}, we obtain that
\begin{equation}\label{ito-y1-x2+x3}
 \begin{aligned}
  &-\langle G\bar{x}(T), x_2^\varepsilon(T)+x_3^\varepsilon(T)\rangle\\
 &\quad=\int_{t_0}^T\{\chi_{E_\varepsilon}(t)(\langle y(t), B(t)\delta u(t)\rangle +\langle Y(t),D(t)\delta u(t)\rangle) +\langle M(t)\bar{x}(t),x_2^\varepsilon(t)+x_3^\varepsilon(t)\rangle\}dt.
 \end{aligned}
\end{equation}
By the definition of the relaxed transposition solution \cite[Theorem 1.1]{W.W.The relaxed transposition solution} to \eqref{SLQ-B-p}, together with \eqref{estimate-x2-x3}, one has that:
\begin{equation}\label{ito-px2-x2-1}
 \begin{aligned}
 &-\langle Gx_2^\varepsilon(T), x_2^\varepsilon(T)\rangle+\int_{t_0}^T\langle M(t)x_2^\varepsilon(t),\ x_2^\varepsilon(t)\rangle dt\\
& =\int_{t_0}^T\chi_{E_\varepsilon}(t)\Big\{\langle \phi(t)D(t)\delta u(t),C(t)x_2^\varepsilon(t)+D(t)\delta u(t)\rangle+\langle \phi(t)C(t)x_2^\varepsilon(t), D(t)\delta u(t) \rangle\Big\}dt+\textit{\textbf{o}}(\varepsilon).
 \end{aligned}
\end{equation}
By substituting \eqref{ito-y1-x2+x3}-\eqref{ito-px2-x2-1} into \eqref{cost functional with estimate}, together with \eqref{Firstly condition}, we obtain that
\begin{equation*}
\begin{aligned}
&\mathcal{J}(t_0,\eta;u^\varepsilon(\cdot))-\mathcal{J}(t_0,\eta;\bar{u}(\cdot))\\
&\hspace{2.8mm}=\textrm{Re}\int_{t_0}^T\chi_{E_\varepsilon}(t)\bigg\{\left\langle R(t)\bar{u}(t)-B(t)^*y(t)-D(t)^*Y(t), \delta u(t)\right\rangle_U \\
&\quad\quad\quad\quad\indent\indent\indent+\frac{1}{2}\langle \left(R(t)-D(t)^*\phi(t)D(t)\right)\delta u(t), \delta u(t)\rangle_U \bigg\}dt+\textit{\textbf{o}}(\varepsilon)\\
&\hspace{2.8mm}=\frac{1}{2}\textrm{Re}\int_{t_0}^T\chi_{E_\varepsilon}(t)\langle \left(R(t)-D(t)^*\phi(t)D(t)\right)\delta u(t), \delta u(t)\rangle_U dt+\textit{\textbf{o}}(\varepsilon).
\end{aligned}
\end{equation*}
Since $\bar{u}(\cdot)$ is an optimal control, $\mathcal{J}(t_0,\eta;u^\varepsilon(\cdot))-\mathcal{J}(t_0,\eta;\bar{u}(\cdot))\geq 0$. Thus,
\begin{equation*}
\frac{1}{2}\textrm{Re}\int_{t_0}^T\chi_{E_\varepsilon}(t)\langle\left(R-D^*\phi D\right)\delta u, \delta u\rangle_Udt\geq \textit{\textbf{o}}(\varepsilon).
\end{equation*}
By Lebesgue differentiation theorem \cite[Theorem 7.10]{Rudin}, for any $t\in[t_0, T)$ and $u(\cdot)\in \mathcal{U}[t_0,T]$, we obtain that
\begin{equation*}
 \textrm{Re} \langle \left(R(t)-D(t)^*\phi(t)D(t)\right) u, u\rangle_U=\lim_{\varepsilon\to0}\frac{1}{\varepsilon}\int_{t}^{t+\varepsilon}\textrm{Re}\langle\left(R-D^*\phi D\right) u,  u\rangle_Ud\tau\geq 0,
\end{equation*}
which gives \eqref{Secondly condition}.
\end{proof}
\begin{rem}
In \cite[Theorem 3.3]{W.W.Pontryagin}, we utilized the spike variation to derive the Pontryagin maximum principle for quantum stochastic control systems, while the convex variation provides an effective method for determining the optimal feedback operator. 
Therefore, Theorem \ref{Pontryagin-type maximum principle for SLQ problem} employs the spike variation and the convex variation. 
\end{rem}

\section{The main result}\label{relationship}
\indent\indent
This section is devoted to investigating the relation between the existence of optimal feedback controls for \textbf{Problem (QSLQ)} and the solvability of the  quantum Riccati equation.

Let $\Theta(\cdot)\in L^2([0,T];\mathcal{L}(L^2(\mathscr{C});U))$ be an optimal feedback operator of \textbf{Problem (QSLQ)}. From \eqref{Firstly condition}, we deduce that
\begin{equation}\label{feedback and By-DY}
R(t)\Theta(t)\bar{x}(t)- B^*(t)y(t)-D^*(t)Y(t)=0,\  {\rm{ a.e.}}\ t\in[t_0,T].
\end{equation}

By Fermion It\^{o}'s formula \cite[Theorem 5.2]{A.H-2}, we can obtain the following result.
\begin{prop}\label{The differentiability of P}
If $P(\cdot)$ is a weak solution to \eqref{quantum Riccati equation}, and $z_1(\cdot)$ and $z_2(\cdot)$ are the solutions to \eqref{to transposition solution QSDE-1} and \eqref{to transposition solution QSDE-2}, respectively, then,  $\langle P(\cdot)z_1(\cdot),z_2(\cdot)\rangle$ is differentiable in $[t_0,T]$, and for any $t\in [t_0,T]$,
\begin{equation*}
\begin{aligned}
d\langle P(t)z_1(t),z_2(t)\rangle&=\langle dP(t)z_1(t),z_2(t)\rangle+\langle P(t)dz_1(t),z_2(t)\rangle+\langle P(t)z_1(t),dz_2(t)\rangle\\
&\hspace{1.1em}+\langle dP(t)dz_1(t),z_2(t)\rangle+\langle dP(t)z_1(t),dz_2(t)\rangle+\langle P(t)dz_1(t),dz_2(t)\rangle.
\end{aligned}
\end{equation*}
\end{prop}
For the need of subsequent proof, let us recall the following result \cite{W.D}.
\begin{lem}\cite[Theorem 4.2]{W.D}\label{measurable}
Let $(\Omega,\mathcal{F})$ be a measurable space. Let $F:(\Omega,\mathcal{F})\to 2^H$ be a closed-valued set mapping such that $F(\omega)\neq\emptyset$ for every $\omega\in \Omega$, and for each open set $O\in H$,
\begin{equation*}
  \{\omega\in \Omega; F(\omega)\cap O\neq\emptyset\}\in \mathcal{F}.
\end{equation*}
Then there is an $H$-valued, $\mathcal{F}$-measurable $f$ such that $f(\omega)\in F(\omega)$ for every $\omega\in \Omega$.
\end{lem}

Next, we  prove the ``if'' part and the ``only if" part of Theorem \ref{the regular solution and the feedback operator}, respectively.
\subsection{Proof of the Necessity of Theorem \ref{the regular solution and the feedback operator}}\label{Sufficiency}
\indent\indent
This subsection is devoted to the proof of ``only if'' part of Theorem \ref{the regular solution and the feedback operator}.

\begin{proof}[{Proof of the Necessity of Theorem \ref{the regular solution and the feedback operator}}]
Let us assume that the equation \eqref{quantum Riccati equation} admits a weak solution $P(\cdot)\in C_\mathbb{A}([t_0,T];\mathcal{L}(L^2(\mathscr{C})))$. Then,
\begin{equation*} -K^{-1} L=-K^{-1}( B^*P+D^*PC)\in L^2([t_0,T];\mathcal{L}(L^2(\mathscr{C});U)).
\end{equation*}
For any $t_0\in[0,T]$, $\xi_1=\xi_2=\eta\in L^2(\mathscr{C}_{t_0})$, choose $z_1(\cdot)=z_2(\cdot)=\bar{x}(\cdot)$, $\mu_1=\mu_2=Bu$ and $\nu_1=\nu_2=Du$ in \eqref{to transposition solution QSDE-1}-\eqref{to transposition solution QSDE-2}. From \eqref{the definition of K and L} and Definition \ref{the def of weak solution to Riccati}, we obtain that
\begin{align*}
&\langle G\bar{x}(T),\bar{x}(T)\rangle+\int_{t_0}^{T}\langle M(t)\bar{x}(t),\bar{x}(t)\rangle dt\\
&=\langle P(t_0)\eta,\eta\rangle+\int_{t_0}^{T}\langle P(t)B(t)u(t),\bar{x}(t)\rangle dt+\int_{t_0}^{T}\langle P(t)\bar{x}(t), B(t)u(t)\rangle dt\\
&\hspace{4.5mm}+\int_{t_0}^{T}\langle P(t)D(t)u(t), D(t)u(t)\rangle dt+\int_{t_0}^{T}\langle K^{-1}(t)L(t)\bar{x}(t),L(t)\bar{x}(t)\rangle dt\\
&\hspace{4.5mm}+\int_{t_0}^{T}\langle P(t)C(t)\bar{x}(t),D(t)u(t)\rangle dt+\int_{t_0}^{T}\langle P(t)D(t)u(t),C(t)\bar{x}(t)\rangle dt.
\end{align*}
Under the case $\Theta(\cdot):=-K^{-1}(\cdot)L(\cdot) $, we deduce that
\begin{equation*}
  \begin{aligned}
    &2\mathcal{J}(t_0,\eta;\Theta(\cdot)\bar{x}(\cdot))\\
     &=\textrm{Re}\left\{\int_{t_0}^T\{\langle M(t)\bar{x}(t), \bar{x}(t)\rangle+\langle R(t)\Theta(t)\bar{x}(t),\Theta(t)\bar{x}(t)\rangle\} dt+\langle  G\bar{x}(T),\bar{x}(T)\rangle\right\}\\
    &=\textrm{Re}\left\{\langle P(t_0)\eta,\eta\rangle+\int_{t_0}^{T}\left\{\langle K^{-1}(t)L(t)\bar{x}(t),\bar{x}(t)\rangle +\langle P(t)D(t)\Theta(t)\bar{x}(t), D(t)\Theta(t)\bar{x}(t)\rangle\right\} dt\right.\\
    &\hspace{12mm}\left. +\int_{t_0}^{T}\Big\{\langle L(t)\bar{x}(t),\Theta(t)\bar{x}(t)\rangle+\langle \Theta(t)\bar{x}(t),L(t)\bar{x}(t)\rangle+\langle R(t)\Theta(t)\bar{x}(t),\Theta(t)\bar{x}(t)\rangle\Big\} dt\right\}\\
    &=\textrm{Re}\langle P(t_0)\eta,\eta\rangle.
  \end{aligned}
\end{equation*}

For any $u(\cdot)\in\mathcal{U}[t_0,T]$, let $x(\cdot)\equiv x(\cdot\ ;t_0,\eta,u(\cdot))$ be the corresponding state process to \eqref{FQSDE-LQ-inital}, we can obtain that
\begin{equation}\label{final of Theorem 1}
\begin{aligned}
   &2\mathcal{J}(t_0,\eta;u(\cdot))\\
    &={\rm{Re}}\left\{\int_{t_0}^T\langle M(t)x(t), x(t)\rangle dt+\int_{t_0}^T\langle R(t)u(t),u(t)\rangle dt+\langle  Gx(T),x(T)\rangle\right\}\\
&={\rm{Re}}\langle P({t_0})\eta,\eta\rangle+{\rm{Re}}\int_{t_0}^{T}\langle L(t)^*K(t)^{-1} L(t)x(t), x(t)\rangle dt\\
&\hspace{1.1em}+2{\rm{Re}}\int_{t_0}^{T}\langle L(t)x(t), u(t)\rangle dt +{\rm{Re}}\int_{t_0}^{T}\langle K(t)u(t),u(t)\rangle dt\\
&=2\mathcal{J}(t_0,\eta;\Theta(\cdot)\bar{x}(\cdot))+{\rm{Re}}\int_{t_0}^{T} \langle K(t)(u(t)+K^{-1}(t)L(t)x(t)),(u(t)+K^{-1}(t)L(t)x(t))\rangle dt.
\end{aligned}
\end{equation}
Therefore,
\begin{equation*}
  \mathcal{J}({t_0},\eta;\Theta(\cdot) \bar{x}(\cdot))\leq \mathcal{J}({t_0},\eta;u(\cdot)), \  u(\cdot)\in\mathcal{U}[{t_0},T],
\end{equation*}
which implies that $\bar{u}(\cdot) =\Theta(\cdot) \bar{x}(\cdot)$ is an optimal control, $\Theta(\cdot)=-K^{-1}(\cdot)L(\cdot)$ is an optimal feedback operator, and \eqref{the relation of value function and P} holds.
Further, by Definition \ref{the def of weak solution to Riccati}, $K(t)> 0$, a.e. $t\in [{t_0},T]$, then
we know that the optimal control is unique.
The proof is complete.
\end{proof}

\subsection{Proof of the Sufficiency of Theorem \ref{the regular solution and the feedback operator}}\label{Necessity}
\indent\indent
This subsection focuses on  proving the ``if'' part of Theorem \ref{the regular solution and the feedback operator}.
\begin{proof}[Proof of the Sufficiency of Theorem \ref{the regular solution and the feedback operator}]
Because the proof process is too long, it can be divided into a few steps.

\textbf{Step 1.}
In this step, we introduce the following four operators and study their strong approximation by the truncated systems.
Let $\Theta(\cdot)$ be an optimal feedback operator of \textbf{Problem (QSLQ)} on $[t_0,T]$, for any $t\in [t_0,T]$, we  define four operators $X(t)$, $\overline{Y}(t)$, $\widetilde{X}(t)$ and $\widetilde{Y}(t)$ on $L^2(\mathscr{C})$ as follows:
\begin{equation}\label{the definition of four operators}
X(t)\varsigma:=x(t;\varsigma),\quad \overline{Y}(t)\varsigma:=y(t;\varsigma),\quad \widetilde{X}(t)\varsigma:=\widetilde{x}(t;\varsigma),\quad \widetilde{Y}(t)\varsigma:=Y(t;\varsigma),\quad  \varsigma\in L^2(\mathscr{C}),
\end{equation}
where $(x(\cdot),y(\cdot),Y(\cdot))$ is the solution to \eqref{FBQSDE} and $\widetilde{x}(\cdot)$ is the solution to \eqref{QSDE-widetilde-x}. Next, we present some properties of the above four operators.

Consider the following  QSDEs:
\begin{equation}\label{FBQSDE-finite}
\left\{
\begin{aligned}
&dX_n(t)=\left\{A(t)+B(t)\Theta(t)\right\}X_n(t)dt+\left\{C(t)+D(t)\Theta(t)\right\}X_n(t)\Upsilon dW(t),\ &{\rm{in}}\ [t_0,T],\\
&d\overline{Y}_n(t)=-\left\{A(t)^* \overline{Y}_n(t)+C(t)^* \widetilde{Y}_n(t)-M(t)X_n(t)\right\}dt+\widetilde{Y}_n(t)\Upsilon dW(t),\ &{\rm{in}}\ [t_0, T],\\
&X_n(t_0)=\Gamma_n,\quad \overline{Y}_n(T)=-GX_n(T),&
\end{aligned}
\right.
\end{equation}
and
\begin{equation}\label{QSDE-widetilde-Xn}
 \left\{
 \begin{aligned}
   d\widetilde{X}_n(t)=&\left\{-A(t)-B(t)\Theta(t)+(C(t)+D(t)\Theta(t))^2\right\}^*\widetilde{X}_n(t)dt\\
  &\hspace{0.5mm}-\{C(t)+D(t)\Theta(t)\}^*\widetilde{X}_n(t)\Upsilon dW(t),\ {\rm{in}}\ [t_0,T],\\
   \widetilde{X}_n(t_0)=&\Gamma_n,
 \end{aligned}
 \right.
\end{equation}
where
$X_n(T)\in\mathcal{L}_2(L^2(\mathscr{C}_T))$ and $\Gamma_n$ is defined in Section \ref{pre}.
Both \eqref{FBQSDE-finite} and \eqref{QSDE-widetilde-Xn} can be regard as $\mathcal{L}_2(L^2(\mathscr{C})))$-valued equations.
By \cite[Lemma 2.2 and Theorem 3.2]{W.W.The relaxed transposition solution}, it is clear that  the equation \eqref{FBQSDE-finite} has
 a unique solution $(X_n(\cdot),\overline{Y}_n(\cdot),\widetilde{Y}_n(\cdot))\in C_\mathbb{A}([t_0,T];\mathcal{L}_2(L^2(\mathscr{C})))\times C_\mathbb{A}([t_0,T];$ $\mathcal{L}_2(L^2(\mathscr{C})))\times L^2_\mathbb{A}([t_0,T];\mathcal{L}_2(L^2(\mathscr{C})))$, and the equation  \eqref{QSDE-widetilde-Xn} admits a unique solution $\widetilde{X}_n(\cdot)\in C_\mathbb{A}([t_0,T];$ $\mathcal{L}_2(L^2(\mathscr{C})))$.
Here $\mathcal{L}_2(L^2(\mathscr{C}))$  denotes the Hilbet space of all Hilbert-Schmidt operators on  $L^2(\mathscr{C})$  with the inner product $\langle X_1, X_2\rangle_{\mathcal{L}_2}={\rm{tr}}(X_1^*X_2)$.

Let $\varsigma, \gamma\in L^2(\mathscr{C})$. It is evident that $x_n(t)=X_n(t)\Gamma_n\varsigma$, $y_n(t)=\overline{Y}_n(t)\Gamma_n\varsigma$, $Y_n(t)=\widetilde{Y}_n(t)\Gamma_n\varsigma$, and $\widetilde{x}_n(t)=\widetilde{X}_n(t)\Gamma_n\gamma$. Thus, $\left(X_n(\cdot)\Gamma_n\varsigma,\overline{Y}_n(\cdot)\Gamma_n\varsigma,\widetilde{Y}_n(\cdot)\Gamma_n\varsigma\right)$ is the solution of \eqref{FBQSDE-xn-yn}, and $\widetilde{X}_n(\cdot)\Gamma_n\gamma$ is the solution of \eqref{QSDE-widetilde-xn}. By \eqref{estimate of FBQSDE-xnynYn} and \eqref{estimate of FQSDE-xn}, we conclude that, for any $t\in [t_0,T]$,
\begin{equation*}
  \begin{array}{ll}
   \|X_n(t)\Gamma_n\varsigma\|_2\leq \mathcal{C}\|\varsigma\|_2,\quad
   &\|\overline{Y}_n(t)\Gamma_n\varsigma\|_2\leq \mathcal{C}\|\varsigma\|_2,\\
 \|\widetilde{X}_n(t)\Gamma_n\gamma\|_2\leq \mathcal{C}\|\gamma\|_2,\quad & \|\widetilde{Y}_n(\cdot)\Gamma_n\varsigma\|_{L^2_\mathbb{A}([t_0,T];L^2(\mathscr{C}))}\leq \mathcal{C}\|\varsigma\|_2,
  \end{array}
\end{equation*}
where the constant $\mathcal{C}$ is independent of $n$. This implies that
\begin{equation}\label{uniformly bounded}
\begin{array}{ll}
\|X_n(t)\Gamma_n\|_{\mathcal{L}(L^2(\mathscr{C});L^2(\mathscr{C}_t))}\leq \mathcal{C},\quad \|\overline{Y}_n(t)\Gamma_n\|_{\mathcal{L}(L^2(\mathscr{C});L^2(\mathscr{C}_t))}\leq \mathcal{C},\\
\|X_n(t)\Gamma_n\|_{\mathcal{L}(L^2(\mathscr{C});L^2(\mathscr{C}_t))}\leq \mathcal{C},\quad \|\widetilde{Y}_n(t)\Gamma_n\|_{\mathcal{L}(L^2(\mathscr{C});L^2_\mathbb{A}([t_0,T];L^2(\mathscr{C})))}\leq \mathcal{C}.
\end{array}
\end{equation}
By Lemma \ref{approximate of FBQSDEs}, we obtain that
\begin{equation}\label{final convergence result}
  \left\{
  \begin{array}{ll}
(s)-\lim\limits_{n\to\infty}X_{n}(t)\Gamma_n\varsigma=X(t)\varsigma,&\quad \textrm{in}\ L^2(\mathscr{C}_t),\\
(s)-\lim\limits_{n\to\infty}\overline{Y}_{n}(t)\Gamma_n\varsigma=\overline{Y}(t)\varsigma, &\quad\textrm{in}\  L^2(\mathscr{C}_t),\\
(s)-\lim\limits_{n\to\infty}\widetilde{X}_{n}(t)\Gamma_n\varsigma=\widetilde{X}(t)\varsigma,&\quad\textrm{in}\ L^2(\mathscr{C}_t),\\
(s)-\lim\limits_{n\to\infty}\widetilde{Y}_{n}(t)\Gamma_n\varsigma\hspace{1mm}=\widetilde{Y}(t)\varsigma,& \quad \textrm{in}\  L^2_\mathbb{A}([t_0,T];L^2(\mathscr{C})).\\
  \end{array}
  \right.
\end{equation}

Applying Fermion It\^{o}'s formula \cite[Theorem 5.2]{A.H-2} to $\langle x(t), \widetilde{x}(t)\rangle$, we get that
\begin{align}
&\langle x(t), \widetilde{x}(t)\rangle-\langle\varsigma, \gamma\rangle\nonumber\\
&=\int_{t_0}^t\langle (A(s)+B(s)\Theta(s))x(s), \widetilde{x}(s)\rangle ds\label{Ito-finite}\\
&\hspace{5mm}-\int_{t_0}^{t}\langle(C(s)+D(s)\Theta(s))x(s),(C(s)+D(s)\Theta(s))^*\widetilde{x}(s)\rangle dt\nonumber\\
&\hspace{5mm}+\int_{t_0}^t\langle x(s), \left\{-A(s)-B(s)\Theta(s)+(C(s)+D(s)\Theta(s))^2\right\}^*\widetilde{x}(s)\rangle dt=0\nonumber.
\end{align}
It follows that
\begin{equation*}
\langle X(t)\varsigma, \widetilde{X}(t)\gamma\rangle=\langle x(t), \widetilde{x}(t)\rangle=\langle\varsigma, \gamma\rangle,
\end{equation*}
which implies that $\widetilde{X}(t)^*X(t)=I$. Thus, for any $t\in [t_0,T]$ and $\gamma\in L^2(\mathscr{C}_{t_0})$, $\widetilde{X}(t)^*X(t)\gamma=\gamma$. Further, by \eqref{the definition of four operators}, there exists a constant $\mathcal{C}>0$ such that
$$\|\widetilde{X}(t)\|_{\mathcal{L}(L^2(\mathscr{C}))}<\mathcal{C},\quad t\in [t_0,T].$$
Hence,
\begin{equation*}
 \|\gamma\|_2=\|\widetilde{X}(t)^*X(t)\gamma\|_2\leq\|\widetilde{X}(t)^*\|_{\mathcal{L}(L^2(\mathscr{C}))}\|X(t)\gamma\|_2 \leq \mathcal{C}\|X(t)\gamma\|_2,\quad t\in [t_0,T].
\end{equation*}
which shows that $X(t)$
is bounded below.
By Lemma \ref{dense}, for any $t\in[t_0,T]$, the range $X(t)$ is dense in $L^2(\mathscr{C}_t)$. 
Therefore,  we conclude that for any $t\in[t_0,T]$, the operator $X(t)$ is invertible, and
\begin{equation}\label{relation of X and widetildeX}
\widetilde{X}(t)^*=X(t)^{-1}.
\end{equation}

\textbf{Step 2.}
In this step, we present an explicit formula of $P(\cdot)$ and give the estimate of the norm of $P(\cdot)$.
By means of \eqref{feedback and By-DY} and \eqref{the definition of four operators}, we find that
\begin{equation}\label{optimal R0X-BY-DZ}
R(t)\Theta(t) X(t)-B(t)^*\overline{Y}(t)-D(t)^*\widetilde{Y}(t)=0, \ {\rm a.e.}\ t\in[t_0,T].
\end{equation}
Put
\begin{equation}\label{the definition of P and Pi}
P(\cdot):=-\overline{Y}(\cdot)\widetilde{X}(\cdot)^*,\quad \Pi(\cdot):=-\widetilde{Y}(\cdot)\widetilde{X}(\cdot)^*.
\end{equation}
It follows from \eqref{relation of X and widetildeX} and \eqref{optimal R0X-BY-DZ} that
\begin{equation}\label{RTheta+B^*P+D^*Pi}
R(t)\Theta(t)+B(t)^*P(t)+D(t)^*\Pi(t)=0,\  \textrm{a.e.}\   t\in [t_0,T].
\end{equation}

Let $s\in[t_0,T]$, $\eta\in L^2(\mathscr{C}_s)$, we consider the following forward-backward QSDE:
\begin{equation}\label{FBQSDE-s}
\left\{
\begin{array}{lr}
dx^s(t)=\left\{A(t)+B(t)\Theta(t)\right\}x^s(t)dt+\left\{C(t)+D(t)\Theta(t)\right\}x^s(t)dW(t),\ &{\rm{in}}\ [s,T],\\
dy^s(t)=-\left\{A(t)^*y^s(t)+C(t)^*Y^s(t)-M(t)x^s(t)\right\}dt+Y^s(t)dW(t),\ &{\rm{in}}\ [s, T],\\
x^s(s)=\eta,\quad y^s(T)=-Gx^s(T).&
\end{array}
\right.
\end{equation}
It follows from Lemma \ref{the result on FBQSDE} and \eqref{feedback and By-DY} that the equation \eqref{FBQSDE-s} admits a unique solution $(x^s(\cdot),y^s(\cdot),Y^s(\cdot))\in$ $C_\mathbb{A}([s,T];L^2(\mathscr{C}))\times C_\mathbb{A}([s,T];L^2(\mathscr{C}))\times L^2_\mathbb{A}([s,T];L^2(\mathscr{C}))$ such that
\begin{equation}\label{estimate of RxByDY}
  R(t)\Theta(t) x^s(t)-B^*(t)y^s(t)-D(t)^*Y^s(t)=0,\ \textrm{a.e.}\ t\in[s,T].
\end{equation}

Next, for each  $t\in [s,T]$,  define two families of operators $X_t^s$ and $\overline{Y}^s_t$ on $L^2(\mathscr{C}_s)$ as follows:
\begin{equation}\label{definition of X_t^s and overline(Y)_t^s}
  X_t^s\eta:=x^s(t;\eta),\quad \overline{Y}_t^s\eta:=y^s(t;\eta).
\end{equation}
Then, from  Lemma \ref{the property of the solution of FQSDE} and \ref{the property of the solution of BQSDE}, one has that
\begin{equation}\label{the estimate XY with s-t}
\left\|X_t^s\eta\right\|_2=\|x^s(t;\eta)\|_2\leq \mathcal{C}\|\eta\|_2,\quad \left\|\overline{Y}_t^s\eta\right\|_2=\|y^s(t;\eta)\|_2\leq \mathcal{C}\|\eta\|_2.
\end{equation}
This implies that $X_t^s,\overline{Y}_t^s\in \mathcal{L}(L^2(\mathscr{C}_s);L^2(\mathscr{C}_t))$ for any $t\in[s,T]$.

From \eqref{definition of X_t^s and overline(Y)_t^s}, it is clear that, for any $\zeta\in L^2(\mathscr{C})$,
\begin{equation*}
X_t^sX(s)\zeta=x^s(t;X(s)\zeta)=x(t,\zeta),\quad
\overline{Y}^s_tX(s)\zeta=y^s(t;X(s)\zeta)=\overline{Y}(t)\zeta.
\end{equation*}
Hence,
\begin{equation}\label{overline(Y)^s_s=overline(Y)(s)X(s)}
\overline{Y}^s_s=\overline{Y}(s)X(s)^{-1}=\overline{Y}(s)\widetilde{X}(s)^*,\  \textrm{for all}\ s\in[t_0,T].
\end{equation}

Let $\eta_1,\eta_2\in L^2(\mathscr{C}_s)$. Then $X^s_t\eta_1=x^s(t;\eta_1)$ and $\overline{Y}^s_t\eta_2=y^s(t;\eta_2)$.
Applying Fermion It\^{o}'s formula \cite[Theorem 5.2]{A.H-2} to $\langle y^s(\cdot;\eta_2), x^s(\cdot;\eta_1)\rangle$ and noting \eqref{FBQSDE-s}-\eqref{estimate of RxByDY}, one has
\begin{equation*}
-\langle GX^s_T\eta_2, X^s_T\eta_1 \rangle =\left\langle \overline{Y}^s_s\eta_2, \eta_1\right\rangle+\int_{s}^{T}\langle M(t)X^s_t\eta_2, X^s_t\eta_1\rangle +\langle R(t)\Theta(t)X^s_t\eta_2,\Theta(t)X^s_t\eta_1 \rangle dt.
\end{equation*}
Therefore,
\begin{equation}\label{to prove self-adjoint-1}
\left\langle\overline{Y}^s_s\eta_2, \eta_1\right\rangle=-\left\langle GX^s_T\eta_2, X^s_T\eta_1\right\rangle-\int_{s}^{T}\left\langle M(t)X^s_t\eta_2, X^s_t\eta_1\right\rangle +\left\langle R(t)\Theta(t)X^s_t\eta_2, \Theta(t)X^s_t\eta_1 \right\rangle dt.
\end{equation}
Analogous to the reasoning presented in \eqref{to prove self-adjoint-1}, we apply Fermion It\^{o}'s formula \cite[Theorem 5.2]{A.H-2} to \(\langle x^s(\cdot;\eta_2), y^s(\cdot;\eta_1)\rangle\) and utilize the self-adjointness of \(G\), \(M(\cdot)\), and \(R(\cdot)\). This leads us to conclude that
\begin{equation}\label{to prove self-adjoint-2}
\left\langle\eta_2, \overline{Y}^s_s\eta_1\right\rangle=-\left\langle GX^s_T\eta_2, X^s_T\eta_1\right\rangle-\int_{s}^{T}\left\langle M(t)X^s_t\eta_2, X^s_t\eta_1\right\rangle +\left\langle R(t)\Theta(t)X^s_t\eta_2, \Theta(t)X^s_t\eta_1 \right\rangle dt.
\end{equation}
From \eqref{to prove self-adjoint-1} and \eqref{to prove self-adjoint-2}, we deduce  that $\overline{Y}_s^s=\overline{Y}(s)\widetilde{X}(s)^*$ is self-adjoint for any $s\in [t_0,T]$. Furthermore, \eqref{the estimate XY with s-t} implies that for any $s\in[t_0,T]$ and $\eta\in L^2(\mathscr{C}_s)$,
\begin{equation}
\|\overline{Y}_s^s\eta\|_2\leq \mathcal{C}\|\eta\|_2,
\end{equation}
where $\mathcal{C}$ is independent of $s\in[t_0,T)$. Therefore, we have that
\begin{equation}\label{the proof of the bounded of YX}
\left\|\overline{Y}(s)\widetilde{X}(s)^*\right\|_{\mathcal{L}(L^2(\mathscr{C}_s))}\leq \mathcal{C}.
\end{equation}
From \eqref{the definition of P and Pi}, \eqref{overline(Y)^s_s=overline(Y)(s)X(s)} and \eqref{the proof of the bounded of YX}, we conclude that $P(\cdot)$ is self-adjoint, and there exist a positive constant $\mathcal{C}$ such that
\begin{equation}\label{show P id bounded}
\|P(s)\|_{\mathcal{L}(L^2(\mathscr{C}_s))}\leq \mathcal{C},\quad s\in[t_0,T].
\end{equation}

\textbf{Step 3.} In this step, we show that $P(\cdot)$ is a solution to an operator-valued differential equation.
For any $t\in [t_0,T]$, define
\begin{equation}\label{definition of finite Pn and Pin}
P_n(t):=-\overline{Y}_{n}(t)\widetilde{X}_{n}(t)^{*},\quad \Pi_n(t):=-\widetilde{Y}_n(t)\widetilde{X}_{n}(t)^{*}.
\end{equation}
By \eqref{final convergence result},
we can obtain that for any $t\in [t_0,T]$ and $\eta, \xi\in L^2(\mathscr{C})$,
\begin{equation}\label{the weak convergence of P}
\begin{aligned}
\lim_{n\to\infty}\left\langle X_{n}(t)\widetilde{X}_{n}(t)^{*}\Gamma_n\eta,\xi\right\rangle=\left\langle X(t)\widetilde{X}(t)^{*}\eta,\xi\right\rangle,\\
\lim_{n\to\infty}\left\langle \overline{Y}_{n}(t)\widetilde{X}_{n}(t)^{*}\Gamma_n\eta,\xi\right\rangle=\left\langle \overline{Y}(t)\widetilde{X}(t)^{*}\eta,\xi\right\rangle,\\
\lim_{n\to\infty}\left\langle \widetilde{Y}_{n}(t)\widetilde{X}_{n}(t)^{*}\Gamma_n\eta,\xi\right\rangle=\left\langle \widetilde{Y}(t)\widetilde{X}(t)^{*}\eta,\xi\right\rangle.
\end{aligned}
\end{equation}
By Fermion It\^{o}'s formula \cite[Theorem 5.2]{A.H-2}, the properties of $\Upsilon$ and \eqref{definition of finite Pn and Pin}, we obtain that
\begin{equation*}
\begin{aligned}
dP_n(t)&=-\left(d\overline{Y}_{n}(t)\right)\widetilde{X}_{n}(t)^{*}-\overline{Y}_{n}(t)\left(d\widetilde{X}_{n}(t)^{*}\right)-d\left(\overline{Y}_{n}(t)\right)d\left(\widetilde{X}_{n}(t)^{*}\right)\\
&=\Big\{A(t)^*\overline{Y}_{n}(t)+C(t)^*\widetilde{Y}_n(t)-M(t)X_n(t)\Big\}\widetilde{X}_{n}(t)^{*}dt-\widetilde{Y}_n(t)\Upsilon dW(t)\widetilde{X}_{n}(t)^{*}\\
 &\hspace{1.1em} -\overline{Y}_{n}(t)\widetilde{X}_{n}(t)^{*}\left\{A(t)+B(t)\Theta(t)-(C(t)+D(t)\Theta(t))^2\right\}dt\\
 &\hspace{1.1em}+\overline{Y}_{n}(t) dW\Upsilon\widetilde{X}_{n}(t)^{*}\{C(t)+D(t)\Theta(t)\}+\widetilde{Y}_n(t)\Upsilon\Upsilon\widetilde{X}_{n}(t)^{*}\{C(t)+D(t)\Theta(t)\} dt\\
&=-\left\{A(t)^* P_n(t)+C(t)^*\Pi_n(t)+M(t)X_n(t)\widetilde{X}_{n}(t)^{*}+P_n(t)B(t)\Theta(t)+P_n(t)A(t)\right.\\
  &\hspace{2.8em}-P_n(t)(C(t)+D(t)\Theta(t))^2+\Pi_n(t)(C(t)+D(t)\Theta(t))\Big\}dt\\
  &\hspace{1.1em}+ \{\Pi_n(t)-P_n(t)(C(t)+D(t)\Theta(t))\}dW(t).
  \end{aligned}
\end{equation*}
Let $\Xi_n(\cdot):=\Pi_n(\cdot)-P_n(\cdot)\{C(\cdot)+D(\cdot)\Theta(\cdot)\}$. Then 
\begin{equation}\label{the expression of P-n}
\left\{
\begin{aligned}
 & dP_n(t)
=-\Big\{A(t)^* P_n(t)+P_n(t)A(t)+C(t)^* \Xi_n(t)+\Xi_n(t)C(t)+C(t)^*P_n(t)C(t)\\
&\hspace{20mm}+\{P_n(t)B(t)+\Xi_n(t)D(t)+C(t)^*P_n(t)D(t)\}\Theta(t)\\
&\hspace{20mm}+M(t)X_n(t)\widetilde{X}_{n}(t)^{*}\Big\}dt+\Xi_n(t)dW(t),\quad \textrm{ in }[t_0,T],\\
  %
  &P_n(T)=G\Gamma_n.
  \end{aligned}
  \right.
\end{equation}
The equation \eqref{the expression of P-n} is regard as $\mathcal{L}_2(L^2(\mathscr{C}))$-valued BQSDE.
By Proposition \ref{The differentiability of P}, we obtain that
\begin{equation*}
\begin{aligned}
 &d\left\langle P_n(t)z_1(t),  z_2(t)\right\rangle \\
  &=-\langle\{P_n(t)B(t)+\Xi_n(t)D(t)+C(t)^*P_n(t)^*D(t)\}\Theta(t)z_1(t),z_2(t)\rangle dt \\
  &\hspace{4.5mm}-\langle M(t)X_n(t)\widetilde{X}(t)^*z_1(t),z_2(t)\rangle dt +\langle P_n(t)\mu_1(t),z_2(t) \rangle dt\\
  &\hspace{4.5mm}+\langle P_n(t)z_1(t),\mu_2(t) \rangle dt+\langle P_n(t)\nu_1(t),C(t)z_2(t)+\nu_2(t) \rangle dt\\
  &\hspace{4.5mm}+\langle P_n(t)C(t)z_1(t),\nu_2(t) \rangle dt+\langle \Xi_n(t)z_1(t),\nu_2(t) \rangle dt+\langle \Xi_n(t)\nu_1(t), z_2(t) \rangle dt.
\end{aligned}
\end{equation*}
Therefore, for any $t\in [t_0,T]$,
\begin{align*}
  & \langle P_n(T)z_1(T),  z_2(T)\rangle+\int_{t}^{T}\langle M(s)X_n(s)\widetilde{X}_n(s)^*z_1(s),z_2(s)\rangle ds\\
  &=\langle P_n(t)z_1(t),  z_2(t)\rangle +\int_{t}^{T}\langle P_n(s)\mu_1(s),z_2(s) \rangle ds+\int_{t}^{T}\langle P_n(s)z_1(s),\mu_2(s) \rangle ds \\
  &\hspace{4.5mm}-\int_{t}^{T}\langle \{P_n(s)B(s)+\Xi_n(s)D(s)+C(s)^*P_n(s)D(s)\}\Theta(s)z_1(s),z_2(s)\rangle ds \\
  &\hspace{4.5mm}+\int_{t}^{T}\langle P_n(s)C(s)z_1(s),\nu_2(t) \rangle ds +\int_{t}^{T}\langle P_n(s)\nu_1(s),C(s)z_2(s)+\nu_2(s) \rangle ds\\
  &\hspace{4.5mm}+\int_{t}^{T}\langle \Xi_n(s)z_1(s),\nu_2(s) \rangle ds+\int_{t}^{T}\langle \Xi_n(s)\nu_1(s), z_2(s) \rangle  ds.
\end{align*}

Furthermore, by \eqref{the weak convergence of P}, we can deduce that
\begin{equation}\label{convergence of P(T)}
\begin{aligned}
&\lim_{n\to\infty}\{\langle P_n(T)z_1(T),  z_2(T)\rangle-\langle P_n(t )z_1(t),  z_2(t )\rangle\}\\
&\hspace{0.9cm}=\langle P(T)z_1(T),  z_2(T)\rangle-\langle P(t)z_1(t),  z_2(t)\rangle.
\end{aligned}
\end{equation}
It follows from  \eqref{final convergence result} and \eqref{definition of finite Pn and Pin} that we conclude that for any $t\in[t_0,T]$,
\begin{equation*}
\begin{aligned}
   & \left|\left\langle \left\{M(t)X_n(t)\widetilde{X}_n(t)^*+(P_n(t)B(t)+\Xi_n(t)D(t)+C(t)^*P_n(t)D(t))\Theta(t)\right\}z_1(t),z_2(t)\right\rangle\right| \\
   & \leq \Big\{\|M(t)\|_{\mathcal{L}(L^2(\mathscr{C}))}\|X_n(t)\|_{\mathcal{L}(L^2(\mathscr{C}))}\|\widetilde{X}_n(t)^*\|_{\mathcal{L}(L^2(\mathscr{C}))}\\
   &\hspace{7mm} +\|C(t)^*\|_{\mathcal{L}(L^2(\mathscr{C}))}\|\overline{Y}_n(t)\|_{\mathcal{L}(L^2(\mathscr{C}))}\|\widetilde{X}_n(t)^*\|_{\mathcal{L}(L^2(\mathscr{C}))}\|D(t)\|_{\mathcal{L}(U;L^2(\mathscr{C}))}\|\Theta(t)\|_{\mathcal{L}(L^2(\mathscr{C}))}\\
  &\hspace{7mm}+\{\|\widetilde{Y}_n(t)\|_{\mathcal{L}(L^2(\mathscr{C}))}\|\widetilde{X}_n(t)^*\|_{\mathcal{L}(L^2(\mathscr{C}))}  +\|\overline{Y}_n(t)\|_{\mathcal{L}(L^2(\mathscr{C}))}\|\widetilde{X}_n(t)^*\|_{\mathcal{L}(L^2(\mathscr{C}))}(\|C(t)\|_{\mathcal{L}(L^2(\mathscr{C}))}\\
    &\hspace{13mm} +\|D(t)\|_{\mathcal{L}(U;L^2(\mathscr{C}))}\|\Theta(t)\|_{\mathcal{L}(L^2(\mathscr{C});U)})\}\|D(t)\|_{\mathcal{L}(U;L^2(\mathscr{C}))}\|\Theta(t)\|_{\mathcal{L}(L^2(\mathscr{C});U)}\\
 &\hspace{7mm}  +\|\overline{Y}_n(t)\|_{\mathcal{L}(L^2(\mathscr{C}))}\|\widetilde{X}_n(t)^*\|_{\mathcal{L}(L^2(\mathscr{C}))}\|B(t)\|_{\mathcal{L}(U;L^2(\mathscr{C}))}\|\Theta(t)\|_{\mathcal{L}(L^2(\mathscr{C});U)}\Big\}
  \|z_1(t)\|_2\|z_2(t)\|_2\\
   & \leq \mathcal{C}\Big\{(1+\|C(t)\|_{\mathcal{L}(L^2(\mathscr{C}))}+\|D(t)\|_{\mathcal{L}(U;L^2(\mathscr{C}))}\|\Theta(t)\|_{\mathcal{L}(L^2(\mathscr{C});U)})\|D(t)\|_{\mathcal{L}(U;L^2(\mathscr{C}))}
   \|\Theta(t)\|_{\mathcal{L}(L^2(\mathscr{C});U)}\\
   &\hspace{9mm} + \|C(t)^*\|_{\mathcal{L}(L^2(\mathscr{C}))}\|D(t)\|_{\mathcal{L}(U;L^2(\mathscr{C}))}\|\Theta(t)\|_{\mathcal{L}(L^2(\mathscr{C});U)} \\
 &\hspace{9mm}  +\|M(t)\|_{\mathcal{L}(L^2(\mathscr{C}))}+ \|B(t)\|_{\mathcal{L}(U;L^2(\mathscr{C}))}\|\Theta(t)\|_{\mathcal{L}(L^2(\mathscr{C});U)}\Big\}
  \|z_1(t)\|_2\|z_2(t)\|_2.
\end{aligned}
\end{equation*}
Similarly,
\begin{equation*}
  \begin{aligned}
&|\langle P_n(t)\mu_1(t),z_2(t) \rangle+\langle P_n(t)z_1(t),\mu_2(t) \rangle\\
&\hspace{4.5mm} +\langle P_n(t)\nu_1(t),C(t)z_2(t)+\nu_2(t) \rangle+\langle P_n(t)C(t)z_1(t),\nu_2(t) \rangle|\\
& \leq \|\overline{Y}_n(t)\|_{\mathcal{L}(L^2(\mathscr{C}))}\|\widetilde{X}_n(t)^*\|_{\mathcal{L}(L^2(\mathscr{C}))}\|\mu_1(t)\|_2\|z_2(t)\|_2\\
&\hspace{4.5mm}+\|\overline{Y}_n(t)\|_{\mathcal{L}(L^2(\mathscr{C}))}\|\widetilde{X}_n(t)^*\|_{\mathcal{L}(L^2(\mathscr{C}))}\|z_1(t)\|_2\|\mu_2(t)\|_2\\
&\hspace{4.5mm}+\|\overline{Y}_n(t)\|_{\mathcal{L}(L^2(\mathscr{C}))}\|\widetilde{X}_n(t)^*\|_{\mathcal{L}(L^2(\mathscr{C}))}\|\nu_1(t)\|_2\{\|C(t)\|_{\mathcal{L}(L^2(\mathscr{C}))}\|z_2(t)\|_2+\|\nu_2(t) \|_2\}\\
&\hspace{4.5mm}+\|\overline{Y}_n(t)\|_{\mathcal{L}(L^2(\mathscr{C}))}\|\widetilde{X}_n(t)^*\|_{\mathcal{L}(L^2(\mathscr{C}))}\|C(t)\|_{\mathcal{L}(L^2(\mathscr{C}))}\|z_1(t)\|_2\|\nu_2(t)\|_2\\
& \leq \mathcal{C}\Big\{\|\mu_1(t)\|_2\|z_2(t)\|_2
+\|\nu_1(t)\|_2\{\|C(t)\|_{\mathcal{L}(L^2(\mathscr{C}))}\|z_2(t)\|_2+\|\nu_2(t) \|_2\}\\
 &\hspace{9mm}  +\|z_1(t)\|_2\|\mu_2(t)\|_2
+\|C(t)\|_{\mathcal{L}(L^2(\mathscr{C}))}\|z_1(t)\|_2\|\nu_2(t)\|_2\Big\},
\end{aligned}
\end{equation*}
and
\begin{align*}
  & |\langle \Xi_n(t)z_1(t),\nu_2(t) \rangle +\langle \Xi_n(t)\nu_1(t), z_2(t) \rangle|\\
& \leq \|\Xi_n(t)\|_{\mathcal{L}(L^2(\mathscr{C}))}\|z_1(t)\|_2\|\nu_2(t)\|_2+\|\Xi_n(t)\|_{\mathcal{L}(L^2(\mathscr{C}))}\|\nu_1(t)\|_2\|z_2(t)\|_2\\
&\leq \|\widetilde{Y}_n(t)\|_{\mathcal{L}(L^2(\mathscr{C}))}\|\widetilde{X}_n(t)^*\|_{\mathcal{L}(L^2(\mathscr{C}))}\|z_1(t)\|_2\|\nu_2(t)\|_2\\
&\hspace{4.5mm}+\|\overline{Y}_n(t)\|_{\mathcal{L}(L^2(\mathscr{C}))}\|\widetilde{X}_n(t)^*\|_{\mathcal{L}(L^2(\mathscr{C}))}\|\nu_1(t)\|_2\|z_2(t)\|_2\\
& \leq \mathcal{C}\Big\{\|z_1(t)\|_2\|\nu_2(t)\|_2+\|\nu_1(t)\|_2\|z_2(t)\|_2\Big\}.
\end{align*}
These, together with  \eqref{relation of X and widetildeX}, \eqref{the weak convergence of P} and Lebesgue's Dominated Convergence Theorem, imply that
 \begin{align}
   &\lim_{n\to\infty}\int_{t }^{T}\langle M( s)X_n(s)\widetilde{X}_n(s)^*z_1(s),z_2(s)\rangle ds\nonumber \\
  &\hspace{0.9cm} =\int_{t }^{T}\lim_{n\to\infty}\langle M(s)X_n(s)\widetilde{X}_n(s)^*z_1(s),z_2(s)\rangle ds\label{convergence of M}\\
  &\hspace{0.9cm} =\int_{t }^{T}\langle M(s)z_1(s),z_2(s)\rangle ds.\nonumber
\end{align}
Similar to \eqref{convergence of M}, we have that
\begin{equation}\label{convergence of P}
  \begin{aligned}
\lim_{n\to\infty}&\left\{\int_{t }^{T}\left\langle\{ (P_n(s)B(s)+C(s)^*P_n(s)^*D(s))\Theta(s)\}z_1(s),z_2(s)\right\rangle dt\right.\\
&\hspace{2.5mm}+\int_{t }^{T}\langle P_n(s)C(s)z_1(s),\nu_2(s)  \rangle ds +\int_{t }^{T}\langle P_n(s)\nu_1(t),C(s)z_2(s)+\nu_2(s) \rangle ds\\
&\hspace{2.5mm}\left.+\int_{t }^{T}\langle P_n(s)\mu_1(s),z_2(s) \rangle ds +\int_{t }^{T}\langle P_n(s)z_1(s),\mu_2(s) \rangle ds\right\}\\
=&\int_{t }^{T}\left\langle\{P(s)B(s)+C(s)^*P(s)^*D(s)\}\Theta(t)z_1(s),z_2(s)\right\rangle ds+\int_{t_0}^{T}\langle P(s)C(s)z_1(s),\nu_2(s)  \rangle ds\\
&\hspace{0.5mm}+\int_{t }^{T}\left\{\langle P(s)\mu_1(s),z_2(s) \rangle+\langle P(s)z_1(s),\mu_2(s) \rangle +\langle P(s)\nu_1(t),C(s)z_2(s)+\nu_2(s) \rangle \right\}ds,
\end{aligned}
\end{equation}
and
\begin{equation}\label{convergence of Xi}
\begin{aligned}
  &\lim_{n\to\infty}\int_{t}^{T}\left\{\langle \Xi_n(s)z_1(s),\nu_2(s) \rangle +\langle \Xi_n(s)\nu_1(s), z_2(s) \rangle+\langle\Xi_n(s)D(s)z_1(s),z_2(s)\rangle\right\}ds  \\
& \hspace{0.9cm}=\int_{t}^{T}\langle \{\Pi(s)-P(s )(C(s)+D(s)\Theta(s))\}z_1(s),\nu_2(s) \rangle ds\\
  &\hspace{13.5mm}+\int_{t}^{T}\langle \{\Pi(s)-P(s)(C(s)+D(s)\Theta(t))\}\nu_1(s), z_2(s) \rangle ds\\
   &\hspace{13.5mm} +\int_{t}^{T}\langle\{\Pi(s)-P(s)(C(s)+D(s)\Theta(t))\}D(s)z_1(s),z_2(s)\rangle ds.
\end{aligned}
\end{equation}
By \eqref{the definition of K and L} and \eqref{RTheta+B^*P+D^*Pi}, we can show that
\begin{equation}\label{the equ of Pi-P(C+DO)}
  \Pi(t)-P(t)(C(t)+D(t)\Theta(t))=0,\ {\rm a.e.}\ t\in[t_0,T].
\end{equation}
It follows from \eqref{convergence of P(T)}-\eqref{the equ of Pi-P(C+DO)} that we have that for any $t\in[t_0,T]$,
  \begin{align}
  & \langle P(T)z_1(T),  z_2(T)\rangle+\int_{t }^{T}\left\langle M(s)z_1(s),z_2(s)\right\rangle ds\nonumber \\
  &\hspace{4.5mm}+\int_{t }^{T}\left\langle P(s)B(s)\Theta(s)z_1(s),z_2(s)\right\rangle ds+\int_{t}^{T}\left\langle C(s)^*P(s)D(t)\Theta(s)z_1(s),z_2(s)\right\rangle ds\nonumber \\
  &=\langle P(t)z_1(t),  z_2(t)\rangle+\int_{t}^{T}\langle P(s)\mu_1(s),z_2(s) \rangle ds+\int_{t}^{T}\langle P(s)z_1(s),\mu_2(s) \rangle ds\label{the ito of P}\\
  &\hspace{4.5mm} +\int_{t }^{T}\langle P(s)C(s)z_1(s),\nu_2(s)\rangle  ds+\int_{t}^{T}\langle P(s)\nu_1(s),C(s)z_2(s)+\nu_2(s) \rangle ds.\nonumber
\end{align}
Then $P(\cdot)$ satisfies the following
\begin{equation}\label{the final of P-Step 3}
\left\{
\begin{aligned}
&dP=-\{A^*P+PA+C^*PC+(PB+C^*PD)\Theta+M\}dt,\ {\rm  in} \ [t_0,T],\\
&P(T)=G.
\end{aligned}
\right.
\end{equation}

\textbf{Step 4.}
In this step, we shall prove that $P(\cdot)$ is the solution to the quantum Riccati equation \eqref{quantum Riccati equation} in the sense of Definition \ref{the def of weak solution to Riccati}.
From \eqref{RTheta+B^*P+D^*Pi}, we can see that
\begin{equation}\label{the step 5}
0=B^*P +R\Theta+D^*P(C+D\Theta)=B^*P +D^*PC+K\Theta, \ \textrm{ a.e.}\ t\in[t_0,T].
\end{equation}
Then
\begin{equation}\label{the relation K}
PB +C^* P D=-\Theta^* K^*.
\end{equation}
Since $K(\cdot)^*=K(\cdot)$, we can infer that
\begin{align}
  & \langle P(T)z_1(T),  z_2(T)\rangle+\int_{t}^T\langle M(s)z_1(s), z_2(s)\rangle ds-\int_t^T \langle\Theta(s)^* K(s)^*\Theta(s)z_1(s), z_2(s)\rangle ds\nonumber\\
   &=\langle P(t)z_1(t),  z_2(t)\rangle +\int_{t}^{T}\langle P(s)\mu_1(s),z_2(s)\rangle ds+\int_{t}^{T}\langle P(s)z_1(s),\mu_2(s)\rangle ds \label{first variation of step 4}\\
   &\hspace{5mm} +\int_{t}^{T}\langle P(s)C(s)z_1(s),\nu_2(s)\rangle ds+\int_{t}^{T}\langle P(s)\nu_1(s), C(s)z_2(s)+\nu_2(s)\rangle ds.\nonumber
\end{align}

For any $t\in [t_0,T]$, $\eta \in L^2(\mathscr{C}_t)$, $u(\cdot)\in L^2([t_0,T];U)$, let $\xi_1=\xi_2=\eta$, $\mu_1=\mu_2=Bu$, $\nu_1=\nu_2=Du$ and $z_1=z_2=x$ in \eqref{to transposition solution QSDE-1}-\eqref{to transposition solution QSDE-2}. Similar to the proof of \eqref{final of Theorem 1}, together with \eqref{first variation of step 4}, noting with \eqref{the definition of K and L} and \eqref{the relation K}, we show that
\begin{equation}\label{cost function of step 4}
 \begin{aligned}
      \mathcal{J}(t,\eta;u(\cdot))
  & =\frac{1}{2}\textrm{Re}\left\{\langle Gx(T), x(T)\rangle+\int_t^T\langle M(s)x(s), x(s)\rangle ds+\int_t^T\langle R(s)u(s),u(s)\rangle_U ds\right\}\\
    & =\frac{1}{2}\textrm{Re}\left\{\langle P(t)\eta, \eta\rangle+\int_t^T\langle K(s)(u(s)-\Theta(s)x(s)),u(s)-\Theta(s)x(s)\rangle_U ds\right\}.
 \end{aligned}
\end{equation}
Hence,
\begin{equation}
\frac{1}{2}\textrm{Re}\langle P(t)\eta, \eta\rangle =\mathcal{J}(t,\eta;\Theta(\cdot)\bar{x}(\cdot))\leq\mathcal{J}(t,\eta;u(\cdot)),\ u(\cdot)\in L^2([t_0,T];U),
\end{equation}
if and only if $K\geq 0$, a.e. $t\in[t_0,T]$.

Put
\begin{gather*}
  \mathfrak{U}_1:=\{t\in(t_0,T); K(t)h=0\ {\rm for\ some\ nonzero} \ h\in U\},\\
  \mathfrak{U}_2:=\{t\in(t_0,T); \|K(t)h\|_U>0, {\rm for\ all}\ h\in \mathcal{B}_U\},
\end{gather*}
where $\mathcal{B}_U:=\{h\in U; \|h\|_U=1\}.$ Clearly, $\mathfrak{U}_1\cap \mathfrak{U}_2=\varnothing$ and $\mathfrak{U}_1\cup\mathfrak{U}_2=(t_0,T)$.
By definition, we have that
\begin{equation*}
 \mathfrak{U}_2=\bigcup_{k=1}^\infty\left\{t\in(t_0,T); \|K(t)h\|_U>\frac{1}{k}, {\rm for\ all}\ h\in  \mathcal{B}_U\right\}.
\end{equation*}
Let $\mathcal{B}^0_U$ be the countable dense subset of $\mathcal{B}_U$. Then
\begin{equation}\label{deng jia ding yi ofU2}
\begin{aligned}
  \mathfrak{U}_2 &=\bigcup_{k=1}^\infty\left\{t\in(t_0,T); \|K(t)h\|_U>\frac{1}{k}, {\rm for\ all}\ h\in  \mathcal{B}^0_U\right\} \\
  &=\bigcup_{k=1}^\infty\bigcap_{h\in \mathcal{B}^0_U}\left\{t\in(t_0,T); \|K(t)h\|_U>\frac{1}{k}\right\}.
\end{aligned}
\end{equation}
Since $K(\cdot)h\in L^2([t_0,T];U),$  we deduce that, for any $h\in U$,
$\left\{t\in (t_0,T); \|K(t)h\|_U>\frac{1}{k}\right\}$
is Lebesgue measurable. Hence both  $\mathfrak{U}_1$ and  $\mathfrak{U}_2$ are Lebesgue measurable.

Now we prove that $K(t)>0$ for a.e. $t\in [t_0,T]$.
Let us use the contradiction argument and assume that this were untrue. Then, $\mathfrak{m}(\mathfrak{U}_1)$, the Lebesgue measure of $\mathfrak{U}_1$, would be positive.

For a.e. $t\in \mathfrak{U}_1$, put
\begin{equation*}
  \mathcal{T}(t):=\{h\in \mathcal{B}_U;\ K(t)h=0\}.
\end{equation*}
Clearly, $\mathcal{T}(t)$ is closed in $U$. Define a map $F:(t_0,T)\to 2^U$ as follows:
\begin{equation*}
  F(t)=\left\{
  \begin{array}{ll}
  \mathcal{T}(t),\ &t\in \mathfrak{U}_1, \\
  0, \ &t\in \mathfrak{U}_2.
\end{array}
\right.
\end{equation*}
Then, $F(t)$ is closed for a.e. $t\in(t_0,T)$.

Let $\mathcal{O}$ be a closed subset of $U$ and $\mathcal{O}_1=\mathcal{O}\cap \mathcal{B}_U$. Put
\begin{equation}\label{definition of sum}
\Omega_1:=\{t\in(t_0,T); F(t)\cap \mathcal{O}\neq \emptyset\},\quad
\Omega_2:=\{t\in(t_0,T); F(t)\cap \mathcal{O}_1\neq \emptyset\}.
\end{equation}
Clearly, $\Omega_2\subset\Omega_1$. Moreover,
\begin{equation*}
\Omega_1=\left\{
\begin{array}{ll}
\Omega_2\cup \mathfrak{U}_2, &0\in \mathcal{O},\\
\Omega_2, &0\notin \mathcal{O}.
\end{array}
\right.
\end{equation*}
Write
\begin{equation*}
  \Omega_3:=\{t\in(t_0,T); \|K(t)h\|_U>0\ {\rm for\ all}\ h\in \mathcal{O}_1\}.
\end{equation*}
It is easy to see that $\Omega_3\cap\Omega_2=\emptyset$ and $\Omega_3\cup\Omega_2=(t_0,T)$.
Furthermore, we can show that $\Omega_1$, $\Omega_2$ and $\Omega_3$ are Lebesgue measurable.

Applying Lemma \ref{measurable} to $F$ with $(t_0,T)$ to find a function $f:[t_0,T]\to U$ such that
\begin{equation*}
f(t)\in F(t)\ {\rm and}\  Kf=0,\ {\rm a.e.}\ t\in(t_0,T).
\end{equation*}
Noting that $\|f(t)\|_U\leq 1$, a.e. $t\in(t_0,T)$,  we find $f\in L^2([t_0,T];U)$. Moreover, we have
\begin{equation*}
  \|f(t)\|_U=1,\ {\rm a.e.}\ t\in\mathfrak{U}_1.
\end{equation*}
Hence, we conclude that $\|f\|_{L^2([t_0,T];U)}>0$.

By \eqref{cost function of step 4}, we see that $\Theta\bar{x}+f$ is also an optimal control. This contradicts the uniqueness of the optimal control. Hence, $\mathfrak{m}(\mathfrak{U}_1)=0$, i.e.
\begin{equation*}
  K(t)\ {\rm is\ injective\ for\ a.e.}\ t\in[t_0,T].
\end{equation*}

Further, we show that $\mathscr{R}(K(t))$ is dense in $U$ for a.e. $t\in [t_0,T]$, where $\mathscr{R}(K(t))$ is the range of $K(t)$. 
Put
\begin{equation*}
  \widetilde{\mathfrak{U}}_1 :=\{t\in(t_0,T);\mathscr{R}(K(t))^\bot\neq0\}, \quad
 \widetilde{\mathfrak{U}}_2 :=\{t\in(t_0,T);\mathscr{R}(K(t))^\bot=0\}.
\end{equation*}
Clearly, $\widetilde{\mathfrak{U}}_1\cup\widetilde{\mathfrak{U}}_2=(t_0,T)$. By the definition of $\widetilde{\mathfrak{U}}_2$, we see that
\begin{equation*}
  \widetilde{\mathfrak{U}}_2=\bigcup_{k=1}^\infty\left\{t\in(t_0,T); {\rm for\ all}\ \widetilde{h}\in\mathcal{B}_U^0,\ {\rm there\ exists}\ h\in\mathcal{B}_U^0\ {\rm such\ that}\ |\langle K(t)h, \widetilde{h}\rangle_U|>\frac{1}{k}\right\}.
\end{equation*}
Then
\begin{equation}\label{widetilde-mathfrakU2}
   \widetilde{\mathfrak{U}}_2=\bigcup_{k=1}^\infty\bigcup_{\widetilde{h}\in\mathcal{B}_U^0}\bigcap_{ h\in\mathcal{B}_U^0}\left\{t\in(t_0,T); \ |\langle K(t)h, \widetilde{h}\rangle_U|>\frac{1}{k}\right\}.
\end{equation}
Since $K(\cdot)h\in L^2([t_0,T];U),$ it follows that for any $h,\widetilde{h}\in U$,
\begin{equation}\label{set of widetildeU}
\left\{t\in(t_0,T); \ |\langle K(t)h, \widetilde{h}\rangle_U|>\frac{1}{k}\right\}
\end{equation}
is Lebesgue measurable. From \eqref{widetilde-mathfrakU2} and \eqref{set of widetildeU}, we can see that both $\widetilde{\mathfrak{U}}_1$ and $\widetilde{\mathfrak{U}}_2$ are also Lebesgue measurable.

To prove that $\mathscr{R}(K(t))$ is  dense in $U$, we use the contradiction argument. If $\mathscr{R}(K(t))$ were not dense in $U$ for a.e. $t\in(t_0,T)$, then $\mathfrak{m}(\widetilde{\mathfrak{U}}_1)>0$.

For a.e. $t\in \widetilde{\mathfrak{U}}_1$, put
\begin{equation*}
 \widetilde{\mathcal{T}}(t):=\left\{\widetilde{h}\in \mathcal{B}_U; \langle K(t)h, \widetilde{h}\rangle=0, {\rm for\ all}\ h\in U\right\}.
\end{equation*}
Clearly, $ \widetilde{\mathcal{T}}(t)$ is closed in $U$.
Define a map $\widetilde{F}:(t_0,T)\to 2^U$ as follows:
\begin{equation*}
  \widetilde{F}(t)=\left\{
  \begin{array}{ll}
\widetilde{  \mathcal{T}}(t),\ &t\in\widetilde{ \mathfrak{U}}_1, \\
  0, \ &t\in \widetilde{\mathfrak{U}}_2.
\end{array}
\right.
\end{equation*}
Then $ \widetilde{F}(t)$ is closed for a.e. $t\in(t_0,T)$.

Similar to \eqref{definition of sum}, let
\begin{equation*}
\widetilde{\Omega}_1:=\{t\in(t_0,T); \widetilde{F}(t)\cap \mathcal{O}\neq \emptyset\},\quad
\widetilde{\Omega}_2:=\{t\in(t_0,T); \widetilde{F}(t)\cap \mathcal{O}_1\neq \emptyset\}.
\end{equation*}
If $0\in \mathcal{O}$, then $\widetilde{\Omega}_1=\widetilde{\Omega}_2\cup \widetilde{\mathfrak{U}}_2$. If $0\notin \mathcal{O}$, then $\widetilde{\Omega}_1=\widetilde{\Omega}_2$. Hence, we only need to
prove that $\widetilde{\Omega}_2$ is measurable. Define
\begin{equation*}
\widetilde{\Omega}_3=\left\{t\in(t_0,T); {\rm for\ all}\ \widetilde{h}\in\mathcal{O}_1, {\rm there\ exists}\ h\in\mathcal{B}_U^0\ {\rm such\ that}\ |\langle K(t)h, \widetilde{h}\rangle_U|>0\right\}.
\end{equation*}
Then $\widetilde{\Omega}_2\cup\widetilde{\Omega}_3=(t_0,T)$ and $\widetilde{\Omega}_2\cap\widetilde{\Omega}_3=\emptyset$. Hence, it
 is enough to prove that $\widetilde{\Omega}_3$ is Lebesgue measurable.
Let $\mathcal{O}_0$ be a countable dense subset of $\mathcal{O}_1$.
Clearly,
\begin{align}
  \widetilde{\Omega}_3&= \bigcup_{k=1}^\infty\left\{t\in(t_0,T); {\rm for\ all}\ \widetilde{h}\in\mathcal{O}_1,  {\rm there\ exists}\ h\in\mathcal{B}_U^0\ {\rm such\ that}\ |\langle K(t)h, \widetilde{h}\rangle_U|>\frac{1}{k}\right\}\nonumber \\
  &=\bigcup_{k=1}^\infty\left\{t\in(t_0,T); {\rm for\ all}\ \widetilde{h}\in\mathcal{O}_0,  {\rm there\ exists}\ h\in\mathcal{B}_U^0\ {\rm such\ that}\ |\langle K(t)h, \widetilde{h}\rangle_U|>\frac{1}{k}\right\} \label{dengjiakehua of widetilde-Omega_3}\\
  &=\bigcup_{k=1}^\infty\bigcap_{\widetilde{h}\in\mathcal{O}_0}\bigcup_{h\in\mathcal{B}_U^0}
  \left\{t\in(t_0,T);  |\langle K(t)h, \widetilde{h}\rangle_U|>\frac{1}{k}\right\}.\nonumber
\end{align}
For any $k\in\mathbb{N}$, $\widetilde{h}\in\mathcal{O}_0$ and $h\in\mathcal{B}_U^0$, noting that $K(\cdot)h\in L^2([t_0,T];U)$,  we deduce that
\begin{equation}\label{final-4-O-3}
\left\{t\in(t_0,T); |\langle K(t)h,\widetilde{h}\rangle_U>\frac{1}{k}\right\}
\end{equation}
is Lebesgue measurable.
From \eqref{dengjiakehua of   widetilde-Omega_3} and \eqref{final-4-O-3}, it follows that $\widetilde{\Omega}_3$ is Lebesgue measurable. Hence, $\widetilde{\Omega}_2$ is also Lebesgue measurable.
%

Now we apply Lemma \ref{measurable} to $\widetilde{F}$ with $(t_0,T)$ to find a function $\widetilde{f}:[t_0,T]\to U$ such that
\begin{equation*}
\widetilde{f}(t)\in \widetilde{F}(t)\ {\rm and}\   \langle K(t)h, \widetilde{f}(t) \rangle_U=0, \  h\in U,\ {\rm a.e.}\ t\in(t_0,T).
\end{equation*}
Since
\begin{equation*}
  \|\widetilde{f}(t)\|_U\leq 1,\ {\rm a.e.}\ t\in(t_0,T),
\end{equation*}
it holds that $\widetilde{f}\in L^2([t_0,T];U)$. Furthermore, it infers that
\begin{equation*}
  \|\widetilde{f}(t)\|_U=1,\ {\rm a.e.}\ t\in \widetilde{\mathfrak{U}}_1,
\end{equation*}
which implies that $\|\widetilde{f}\|_{L^2([t_0,T];U)}>0$.

We prove that $\Theta\bar{x}+\widetilde{f}$ is also an optimal control. Indeed, by the choice of $\widetilde{f}$, it holds that for any $t\in[t_0,T)$,
\begin{equation*}
  \int_{t_0}^{T}\langle K(u-\Theta\bar{x}-\widetilde{f}),\widetilde{f}\rangle_U ds=0,
\end{equation*}
and
\begin{equation*}
  \int_{t_0}^{T}\langle K\widetilde{f}, u-\Theta\bar{x}-\widetilde{f}\rangle_U ds= \int_{t}^{T}\langle\widetilde{f}, K(u-\Theta\bar{x}-\widetilde{f})\rangle_U ds=0.
\end{equation*}
Therefore, for any $u(\cdot)\in L^2([t_0,T];U)$,
\begin{equation}\label{Thetabarx+widetildef is also an optimal control}
\begin{aligned}
& \int_{t_0}^T\langle K(u-\Theta\bar{x}-\widetilde{f}),u-\Theta\bar{x}-\widetilde{f}\rangle_U ds\\
 &\hspace{7mm} =\int_{t_0}^T\langle K(u-\Theta\bar{x}-\widetilde{f}),u-\Theta\bar{x}\rangle_U ds\\
&\hspace{7mm}=\int_{t_0}^T\langle K(u-\Theta\bar{x}),u-\Theta\bar{x}\rangle_U ds.
  \end{aligned}
\end{equation}
By \eqref{cost function of step 4} and \eqref{Thetabarx+widetildef is also an optimal control}, we obtain that
\begin{equation*}
  \mathcal{J}(t_0,\eta; \Theta(\cdot)\bar{x}(\cdot)-\widetilde{f}(\cdot))\leq  \mathcal{J}(t_0,\eta;u(\cdot)),\quad u(\cdot)\in L^2([t_0,T];U).
\end{equation*}
This indicates that $\Theta\bar{x}+\widetilde{f}$ is also an optimal control, which contradicts the uniqueness of the optimal controls. Hence, $\mathscr{R}(K(t))$ is dense in $U$ for a.e. $t\in(t_0,T)$.
For a.e. $t\in[t_0,T]$ and any $u\in U$,  $K(t)^{-1}:\mathscr{R}(K(t))\to U$, $K(t)^{-1}K(t)u=u$, which
is densely defined in $U$.

We prove that $K(t)^{-1}$ is a closed operator for a.e. $t\in[t_0,T]$. Since $K(t)$ is bounded for a.e. $t\in[t_0,T]$, it follows that  $K(t)$ is a closed operator for a.e. $t\in[t_0,T]$, and thus $K(t)^{-1}$ is also a closed operator for a.e. $t\in[t_0,T]$.
Hence,  Definition \ref{the def of weak solution to Riccati} \textbf{(i)} holds.

From \eqref{the step 5},
\begin{equation}\label{the final of step 4}
  -K^{-1}(B^*P+D^*PC)=\Theta.
\end{equation}
By substituting \eqref{the final of step 4} into \eqref{first variation of step 4}, we can demonstrate that Definition \ref{the def of weak solution to Riccati} \textbf{(ii)} holds.

\textbf{Step 5.} Finally, we prove the uniqueness of the weak solutions to \eqref{quantum Riccati equation}.
 Assume $P_1(\cdot),P_2(\cdot)\in C_\mathbb{A}([t_0,T];\mathcal{L}(L^2(\mathscr{C})))$ are two weak solution to \eqref{quantum Riccati equation}. By \eqref{final of Theorem 1}, one gets that
\begin{align}
  & \int_{t}^T\langle M(s)x(s),x(s)\rangle ds+ \int_{t}^T\langle R(s)u(s),u(s)\rangle_U ds+\langle Gx(T),x(T)\rangle\nonumber\\
  &\hspace{6mm} =\langle P_1(t)\eta,\eta\rangle+\int_{t}^{T} \langle K_1(s)(u(s)-\Theta(s)x(s)),u(s)-\Theta(s)x(s)\rangle ds\label{uniqueness-1}\\
&\hspace{6mm} =\langle P_2(t)\eta,\eta\rangle +\int_{t}^{T} \langle K_2(s)(u(s)-\Theta(s)x(s)),u(s)-\Theta(s)x(s)\rangle ds,\nonumber
\end{align}
where $K_i=R+D^*P_iD$ for $i=1,2$. Taking $u=\Theta x$ in \eqref{uniqueness-1},
 we deduce that for any $t\in[t_0,T]$,
 \begin{equation}\label{P1-P2}
\langle P_1(t)\eta,\eta\rangle=\langle P_2(t)\eta,\eta\rangle,\ \eta\in L^2(\mathscr{C}_{t }).
 \end{equation}
Hence, for any $\xi,\eta\in L^2(\mathscr{C}_{t})$, we have that
\begin{equation*}
\langle P_1(t)(\xi+\eta),\xi+\eta\rangle=\langle P_2(t)(\xi+\eta),\xi+\eta\rangle,
\end{equation*}
and
\begin{equation*}
\langle P_1(t)(\eta-\xi),\eta-\xi\rangle=\langle P_2(t)(\eta-\xi),\eta-\xi\rangle.
\end{equation*}
These, together with $P_1(\cdot)=P_1(\cdot)^*$ and $P_2(\cdot)=P_2(\cdot)^*$, imply that
\begin{equation*}
 \langle P_1(t)\eta,\xi\rangle=\langle P_2(t)\eta,\xi\rangle,\quad \xi,\eta\in L^2(\mathscr{C}_{t}).
\end{equation*}
Hence, $P_1(t)=P_2(t)$ for any $t\in [t_0,T]$. Consequently, the uniqueness of the desired solution follows.
Then, we complete the proof of the ``if " part.
\end{proof}

\section{Appendix}
\indent\indent
To prove Lemma \ref{approximate of FBQSDEs}, we first present the following auxiliary result.
\begin{lem}\label{cor of martingale representation}
Let $W(\cdot)$ be the Fermion Brownian motion defined by \eqref{Fermion Brownian motion}. Then, for any $f\in L^1_\mathbb{A}([0,T];$ $L^2(\mathscr{C}))$, there exists a unique $\mathcal{K}(\cdot,\cdot)\in  L^1([0,T];L^2_\mathbb{A}([0,T];L^2(\mathscr{C})))$ satisfying the following conditions:
\begin{description}
 \item[(i)] $\mathcal{K}(s,\sigma)=0, \quad\sigma >s;$
  \item[(ii)] For any $s\in [0,T],$
  \begin{equation}\label{cor of martingale f}
    f(s)=m(f(s))+\int_0^s\mathcal{K}(s,\sigma)dW(\sigma);
  \end{equation}
  \item[(iii)]
  \begin{equation}\label{estimate of cor of martingale martingale}
  \|\mathcal{K}\|_{L^1([0,T];L^2_\mathbb{A}([0,T];L^2(\mathscr{C})))}\leq 2\|f\|_{L^1_\mathbb{A}([0,T];L^2(\mathscr{C}))}.
\underline{}  \end{equation}
\end{description}
\end{lem}
\begin{proof}
For any $f\in L^1_\mathbb{A}([0,T];L^2(\mathscr{C}))$, we can find a sequence of simple processes $\{f_n\}_{n=1}^\infty$ such that $$\lim_{n\to\infty}\|f_n-f\|_{ L^1_\mathbb{A}([0,T];L^2(\mathscr{C}))} =0,$$
where $f_n(t)=\sum_{i=0}^{n-1}\chi_{[t_i^n,t_{i+1}^n)}(t)a_{t_i^n}$ for partitions $\{t_i^n\}_{i=0}^n$ of $[0,T]$, and $a_{t_i^n}\in L^2(\mathscr{C}_{t_i^n})$.

By the noncommutative martingale representation theorem \cite[Theorem 4.1]{B.S.W.2},  for any $a_{t_i^n}$, there is a $k_{t_i^n}\in L^2_\mathbb{A}([0,T];L^2(\mathscr{C}))$ such that
\begin{equation}\label{martingale of simple}
a_{t_i^n}=m(a_{t_i^n})+ \int_0^{t_i^n}k_{t_i^n}(\sigma)dW(\sigma); \quad k_{t_i^n}(\sigma)=0,\ \sigma>t_i^n.
\end{equation}
For any $(s,\sigma)\in [0,T]\times [0,T]$, put $$\mathcal{K}_n(s,\sigma):=\sum_{i=0}^{n-1}\chi_{[t_i^n,t_{i+1}^n)}(s)\chi_{[0,s]}(\sigma)k_{t_i^n}(\sigma).$$ Then
\begin{equation}\label{the representation of f_n}
  f_n(s)=m(f_n(s))+\int_0^s\mathcal{K}_n(s,\sigma)dW(\sigma).
\end{equation}
Hence, for any $i,j\in \mathbb{N}$,
\begin{align}\label{cauchy estimate zuocha}
  &\int_0^T\left\{\int_0^T\|\mathcal{K}_i(s,t)-\mathcal{K}_j(s,t)\|_2^2dt\right\}^{\frac{1}{2}}ds\nonumber \\
   & \qquad=\int_0^T\left\{\int_0^s\|\mathcal{K}_i(s,t)-\mathcal{K}_j(s,t)\|_2^2dt\right\}^{\frac{1}{2}}ds \nonumber \\
  &\qquad=\int_0^T\left\|\int_0^s\{\mathcal{K}_i(s,t)-\mathcal{K}_j(s,t)\}dW(t)\right\|_2ds  \\
  &\qquad=\int_0^T\|f_i(s)-f_j(s)-m(f_i(s)-f_j(s))\|_2ds\nonumber \\
   &\qquad\leq2\|f_i-f_j\|_{L^1_\mathbb{A}([0,T];L^2(\mathscr{C}))}.\nonumber
\end{align}
Since $\{f_n\}_{n=1}^\infty$ is a Cauchy sequence in $L^1_\mathbb{A}([0,T];L^2(\mathscr{C}))$, by \eqref{cauchy estimate zuocha}, we deduce that $\{\mathcal{K}_n\}_{n=1}^\infty$ is a Cauchy sequence in $L^1([0,T];L^2_\mathbb{A}([0,T];L^2(\mathscr{C}))$. Thus, there is a $\mathcal{K}\in L^1([0,T];L^2_\mathbb{A}([0,T];L^2(\mathscr{C}))$ so that $\lim\limits_{n\to \infty}\mathcal{K}_n=\mathcal{K}$ in $L^1([0,T];L^2_\mathbb{A}([0,T];L^2(\mathscr{C}))$. Combining with \eqref{the representation of f_n}, we can obtain that \eqref{cor of martingale f}.

Similar to \eqref{cauchy estimate zuocha}, one has
\begin{equation}\label{cauchy estimate}
\int_0^T\left\{\int_0^T\|\mathcal{K}_n(s,t)\|_2^2dt\right\}^{\frac{1}{2}}ds\leq 2\|f_n\|_{L^1_\mathbb{A}([0,T];L^2(\mathscr{C}))}.
\end{equation}
Taking $n\to \infty$ in \eqref{cauchy estimate}, \eqref{estimate of cor of martingale martingale} holds.
\end{proof}
\begin{proof}[\textbf{Proof of Lemma \ref{approximate of FBQSDEs}}]

First, we prove that
$$\lim\limits_{n\to\infty}x_n(\cdot)=x(\cdot)\ {\rm{ in}}\ C_\mathbb{A}([t_0,T]; L^2(\mathscr{C})).$$
It follows from \eqref{FBQSDE} and \eqref{FBQSDE-finite} that for any $t\in [t_0,T]$,
\begin{gather*}
x(t)=\varsigma+\int_{t_0}^{t}\{A(s)+B(s)\Theta(s)\}x(s)ds+\int_{t_0}^{t}\{C(s)+D(s)\Theta(s)\}x(t)dW(s),\\
x_n(t)=\Gamma_n\varsigma+\int_{t_0}^{t}\{A(s)+B(s)\Theta(s)\}x_n(s)ds+\int_{t_0}^{t}\{C(s)+D(s)\Theta(s)\}x_n(s)dW(s).
\end{gather*}
By Minkowski's inequality, we can obtain that
\begin{equation}\label{the first estimate of x-xn}
\begin{aligned}
&\sup_{t\in[t_0,T]}\|x_n(t)-x(t)\|_2^2\\
&\leq 3\sup_{t\in[t_0,T]}\left\{\left\|\Gamma_n\varsigma-\varsigma\right\|_2^2+\left\|\int_{t_0}^{t}\{A(s)+B(s)\Theta(s)\}\left\{x_n(s)-x(s)\right\}ds\right\|_2^2\right.\\
&\hspace{22mm}\left.+\left\|\int_{t_0}^{t}\{C(s)+D(s)\Theta(s)\}\left\{x_n(s)-x(s)\right\}dW(s)\right\|_2^2\right\}.\\
&\leq 3\left\{\left\|\Gamma_n\varsigma-\varsigma\right\|_2^2+\left\|\int_{t_0}^{T}\{A(s)+B(s)\Theta(s)\}\left\{x_n(s)-x(s)\right\}ds\right\|_2^2\right.\\
&\hspace{9.5mm}\left.+\left\|\int_{t_0}^{T}\{C(s)+D(s)\Theta(s)\}\left\{x_n(s)-x(s)\right\}dW(s)\right\|_2^2\right\}.\\
\end{aligned}
\end{equation}

Next, we calculate each item of the right of the above inequality.
By H\"{o}lder's inequality, we have that
\begin{align}
&\left\|\int_{t_0}^{T}\{A(s)+B(s)\Theta(s)\}\left\{x_n(s)-x(s)\right\}ds\right\|_2^2\nonumber\\
&\hspace{9mm}\leq\left(\int_{t_0}^{T}\|\{A(s)+B(s)\Theta(s)\}\left\{x_n(s)-x(s)\right\}\|_2 ds\right)^2\label{estimate of A+B0}\\
&\hspace{9mm}\leq\left(\int_{t_0}^{T}\|\{A(s)+B(s)\Theta(s)\}\|_{\mathcal{L}(L^2(\mathscr{C}))} ds\right)^2\left\{\sup_{t\in[t_0,T]}\left\|x_n(t)-x(t)\right\|_2^2\right\}.\nonumber
\end{align}
By the isometry of the It\^{o}-Clifford stochastic integral in $L^2(\mathscr{C})$ \cite[Theorem 3.5(c)]{B.S.W.1}, 
\begin{align}
&\left\|\int_{t_0}^{T}\{C(s)+D(s)\Theta(s)\}\left\{x_n(s)-x(s)\right\}dW(s)\right\|_2^2\nonumber\\
&\hspace{9mm}\leq\int_{t_0}^{T}\|C(s)+D(s)\Theta(s)\|^2_{\mathcal{L}(L^2(\mathscr{C}))}\|x_n(s)-x(s)\|_2^2ds\label{estimate of C+D0}\\
&\hspace{9mm}\leq\int_{t_0}^{T}\|C(s)+D(s)\Theta(s)\|^2_{\mathcal{L}(L^2(\mathscr{C}))}ds\left\{\sup_{t\in[t_0,T]}\left\|x_n(t)-x(t)\right\|_2^2\right\}.\nonumber
\end{align}

Taking the limit as $ n \to \infty$ on both sides of \eqref{the first estimate of x-xn}, and using \eqref{estimate of A+B0} and \eqref{estimate of C+D0}, we obtain that
\begin{equation*}
\begin{aligned}
&\lim_{n\to\infty}\sup_{t\in[t_0,T]}\|x_n(t)-x(t)\|_2^2  \\
  & \leq 3\lim_{n\to\infty}\left\|\Gamma_n\varsigma-\varsigma\right\|_2^2+3\lim_{n\to\infty}\left(\int_{t_0}^{T}\|\{A(s)+B(s)\Theta(s)\}\|_{\mathcal{L}(L^2(\mathscr{C}))} ds\right)^2\left\{\sup_{t\in[t_0,T]}\left\|x_n(t)-x(t)\right\|_2^2\right\}\\
  &\hspace{4mm}+3\lim_{n\to\infty}\int_{t_0}^{T}\|C(s)+D(s)\Theta(s)\|^2_{\mathcal{L}(L^2(\mathscr{C}))}ds\left\{\sup_{t\in[t_0,T]}\left\|x_n(t)-x(t)\right\|_2^2\right\}.
\end{aligned}
\end{equation*}
By the definition of $\Gamma_n$, we observe that
\begin{equation}\label{the Gamman  I  estimate about convergence of x}
\lim_{n\to\infty}\left\|\Gamma_n\varsigma-\varsigma\right\|_2^2=0.
\end{equation}
Then, by Gronwall's inequality, we obtain that
\begin{equation}\label{the final convergence of xn-x}
  \lim_{n\to\infty}\sup_{t\in[t_0,T]}\|x_n(t)-x(t)\|_2^2=0.
\end{equation}

Similar to the proof of \eqref{the final convergence of xn-x},
$ \lim\limits_{n\to\infty}\sup\limits_{t\in[t_0,T]}\|\widetilde{x}_n(t)-\widetilde{x}(t)\|_2^2=0$ holds.
Next, we prove that
\begin{equation*}
\lim_{n\to\infty}y_n(\cdot)=y(\cdot)\ \textrm{in}\  C_\mathbb{A}([t_0,T]; L^2(\mathscr{C})),
\end{equation*}
and
\begin{equation*}
  \lim_{n\to\infty}Y_n(\cdot)=Y(\cdot)\ \textrm{in}\ L^2_\mathbb{A}([t_0,T]; L^2(\mathscr{C})).
\end{equation*}
Let
\begin{equation}\label{the sup of exp(At) and AC}
  \mathcal{M}:=\sup_{t\in[t_0,T]}\left\{\|A(t)\|_{\mathcal{L}(L^2(\mathscr{C}))}+\|C(t)\|_{\mathcal{L}(L^2(\mathscr{C}))}\right\}.
\end{equation}
Put
\begin{equation*}
\widetilde{T}_1:=\inf\left\{t\in[t_0,T];\, \left[ 6(5\mathcal{M}+1)\mathcal{M}+5\mathcal{M}\right]\cdot\max\{(T-t)^2,T-t\}\leq \frac{1}{2}\right\}.
\end{equation*}
Recall that for any $t\in[t_0,T]$,
\begin{equation}\label{representation of solution of BQSDE}
y(t)=-Gx(T)+\int_t^T\left\{A(s)^*y(s)+C(s)^*Y(s)-M(s)x(s)\right\}ds-\int_t^TY(s)dW(s).
\end{equation}
From \eqref{the condition of ABCD} and \eqref{condition of MGR}, it follows that $A(\cdot)^*y(\cdot)+C(\cdot)^*Y(\cdot)-M(\cdot)x(\cdot)\in L^1_\mathbb{A}([t_0,T];L^2(\mathscr{C})).$ Then, by Lemma \ref{cor of martingale representation}, we can find that there is an
$\mathcal{K}(\cdot,\cdot)\in L^1([t_0,T];L^2_\mathbb{A}([t_0,T];L^2(\mathscr{C})))$ such that
\begin{equation}\label{the condition of K}
\begin{aligned}
&A(s)^*y(s)+C(s)^*Y(s)-M(s)x(s)\\
&=m(A(s)^*y(s)+C(s)^*Y(s)-M(s)x(s))+\int_{t_0}^{s}\mathcal{K}(s,\tau)dW(\tau),
 \end{aligned}
\end{equation}
and
\begin{equation}\label{the estimate of K}
\|\mathcal{K}(\cdot,\cdot)\|_{L^1([t_0,T];L^2_\mathbb{A}([t_0,T];L^2(\mathscr{C})))}\leq 2\|A^*y+C^*Y-Mx\|_{L^1_\mathbb{A}([t_0,T];L^2(\mathscr{C}))}.
\end{equation}

By the noncommutative martingale representation theorem \cite[Theorem 4.1]{B.S.W.2}, there exists an $l\in L^2_\mathbb{A}([t_0,T];L^2(\mathscr{C}))$ such that
\begin{equation}\label{reresentation of l}
m(y_T|L^2(\mathscr{C}_t))=m(y_T)+\int_{t_0}^tl(\sigma)dW(\sigma),\quad t\in[t_0,T].
\end{equation}
Put $$y(t):=m\left(y_T-\int_t^T\{A(s)^*y(s)+C(s)^*Y(s)-M(s)x(s)\}ds\bigg|L^2(\mathscr{C}_t)\right),\ t\in[t_0,T].$$
Similar to classical stochastic Fubini theorem \cite[Theorem 2.141]{L.Z}, we can prove that if $\mathcal{K}(\cdot,\cdot)\in L^1([t_0,T];$ $L^2_\mathbb{A}([t_0,T];L^2(\mathscr{C}))),$ then $$\int_{t_0}^T\left\|\int_{t_0}^T\mathcal{K}(s,\sigma)dW(\sigma)\right\|_2ds<\infty,$$ and
\begin{equation*}
\int_{t_0}^T\int_{t_0}^T\mathcal{K}(s,\sigma)dW(\sigma)ds=\int_{t_0}^T\int_{t_0}^T\mathcal{K}(s,\sigma)dsdW(\sigma). 
\end{equation*}
Then, by \eqref{the condition of K} and \eqref{reresentation of l},  we deduce that
\begin{align*}
y(t)=y_T&-\int_t^T\{A(s)^*y(s)+C(s)^*Y(s)-M(s)x(s)\}ds\\
&-\int_t^Tl(\sigma)dW(\sigma)+\int_t^T\int_t^s\mathcal{K}(s,\sigma)dW(\sigma)ds\\
=y_T&-\int_t^T\{A(s)^*y(s)+C(s)^*Y(s)-M(s)x(s)\}ds\\
&-\int_t^Tl(\sigma)dW(\sigma)+\int_t^T\int_s^T\mathcal{K}(\sigma, s)d\sigma dW(s).
\end{align*}
This, together with \eqref{representation of solution of BQSDE}, implies that
\begin{equation}\label{the representation of K}
Y(s)=l(s)-\int_{s}^{T}\mathcal{K}(\sigma,s)d\sigma.
\end{equation}

Similarly, we can obtain that
\begin{equation*}
Y_n(s)=l_n(s)-\int_{s}^{T}\mathcal{K}_n(\sigma,s)d\sigma,
\end{equation*}
and
\begin{equation*}
y_n(t)=-Gx_n(T)+\int_t^T\left\{A(s)^* y_n(s)+C(s)^* Y_n(s)-M(s)x_n(s)\right\}ds-\int_t^TY_n(s)dW(s),
\end{equation*}
where $l_n(\cdot)\in L^2_\mathbb{A}([t_0,T];L^2(\mathscr{C}))$ such that
\begin{equation} \label{the representation of ln}
  m\left(Gx_n(T)\big|L^2(\mathscr{C}_t)\right)=m(Gx_n(T))+\int_{t_0}^tl_n(s)dW(s),
\end{equation}
and $\mathcal{K}_n(\cdot,\cdot)\in L^1([t_0,T];L^2_\mathbb{A}([t_0,T];L^2(\mathscr{C})))$ such that
\begin{equation}\label{the representation of Kn}
\begin{aligned}
&A(s)^* y_n(s)+C(s)^* Y_n(s)-M(s)x_n(s)\\
&=m\left(A(s)^* y_n(s)+C(s)^* Y_n(s)-M(s)x_n(s)\right)+\int_{t_0}^{s}\mathcal{K}_n(s,\tau)dW(\tau),
\end{aligned}
\end{equation}
and
\begin{equation}\label{the estimate of Kn}
\|\mathcal{K}_n(\cdot,\cdot)\|_{L^1([t_0,T];L^2_\mathbb{A}([t_0,T];L^2(\mathscr{C})))}\leq \mathcal{C}\left\|A^* y_n+C^* Y_n-Mx_n\right\|_{L^2_\mathbb{A}([t_0,T];L^2(\mathscr{C}))}.
\end{equation}

For any $t\in[\widetilde{T}_1,T]$, we have
\begin{equation}\label{yn-y+Y-Yn}
\begin{aligned}
&\sup_{t\in[\widetilde{T}_1,T]}\|y_n(t)-y(t)\|^2_2+(5\mathcal{M}+1)\int_{\widetilde{T}_1}^T\|Y_n(t)-Y(t)\|_2^2dt\\
&\leq(10\mathcal{M}+2)\left\{\int_{\widetilde{T}_1}^T\left\|l_n(t)-l(t)\right\|_2^2dt+\int_{\widetilde{T}_1}^T\left\|\int_{t}^{T}\mathcal{K}_n(s,t)-\mathcal{K}(s,t)ds\right\|^2_2dt\right\}\\
&\quad+5\|Gx(T)-Gx_n(T)\|_2^2+5\int_{\widetilde{T}_1}^{T}\|Y_n(s)-Y(s)\|_2^2ds \\
&\quad+5(T-\widetilde{T}_1)\sup_{t\in [\widetilde{T}_1,T]}\left\{\int_{t}^{T}\|A(s)^*y_n(s)-A(s)^*y(s)\|_2^2ds+\int_{t}^{T}\|C(s)^*Y_n(s)-C(s)^*Y(s)\|_2^2ds\right.\\
&\hspace{3.9cm}\left.+\int_{t}^{T}\|M(s)x_n(s)-M(s)x(s)\|_2^2ds\right\}.\\
&\leq(10\mathcal{M}+2)\left\{\int_{\widetilde{T}_1}^T\left\|l_n(t)-l(t)\right\|_2^2dt+\int_{\widetilde{T}_1}^T\left\|\int_{t}^{T}\mathcal{K}_n(s,t)-\mathcal{K}(s,t)ds\right\|^2_2dt\right\}\\
&\quad+5\|Gx(T)-Gx_n(T)\|_2^2+5\int_{\widetilde{T}_1}^{T}\|Y_n(s)-Y(s)\|_2^2ds \\
&\quad+5(T-\widetilde{T}_1)\left\{\int_{\widetilde{T}_1}^{T}\|A(s)^*y_n(s)-A(s)^*y(s)\|_2^2ds+\int_{\widetilde{T}_1}^{T}\|C(s)^*Y_n(s)-C(s)^*Y(s)\|_2^2ds\right.\\
&\hspace{2.8cm}\left.+\int_{\widetilde{T}_1}^{T}\|M(s)x_n(s)-M(s)x(s)\|_2^2ds\right\}.
\end{aligned}
\end{equation}
By \eqref{the final convergence of xn-x}, we have that
\begin{equation}\label{G_nx_nT-Gx_nT}
\lim_{n\to\infty}\|Gx(T)-Gx_n(T)\|_2^2=0,
\end{equation}
and
  \begin{align*}
&\int_{\widetilde{T}_1}^{T}\|M(s)x_n(s)-M(s)x(s)\|_2^2ds\\
&\hspace{8mm}\leq \int_{\widetilde{T}_1}^{T}\|M(s)\|^2_{\mathcal{L}(L^2(\mathscr{C}))}\|x_n(s)-x(s)\|_2^2ds\\
&\hspace{8mm}\leq \int_{\widetilde{T}_1}^{T}\|M(s)\|^2_{\mathcal{L}(L^2(\mathscr{C}))}ds\left\{\sup_{t\in[\widetilde{T}_1,T]}\|x_n(t)-x(t)\|_2^2\right\}\\
&\hspace{8mm}=0.
\end{align*}

Besides that,
\begin{align}
&\left(\int_{\widetilde{T}_1}^{T}\left\|A(s)^*(y_n(s)-y(s))\right\|_2ds\right)^2+\left(\int_{\widetilde{T}_1}^{T}\left\|C(s)^*(Y_n(s)-Y(s))\right\|_2ds\right)^2\nonumber\\
&\hspace{10mm}\leq \int_{\widetilde{T}_1}^{T}\|A(s)^*\|_{\mathcal{L}(L^2(\mathscr{C}))}^2ds\left\{\sup_{t\in[\widetilde{T}_1,T]}\|y(t)-y_n(t)\|_2^2\right\}\label{(y_n(s)-y(s))+(Y_n(s)-Y(s))}\\
&\hspace{14mm}+\int_{\widetilde{T}_1}^{T}\|C(s)^*\|_{\mathcal{L}(L^2(\mathscr{C}))}^2ds\int_{\widetilde{T}_1}^{T}\|Y_n(s)-Y(s)\|_2^2ds.\nonumber
\end{align}
From \eqref{reresentation of l} and \eqref{the representation of ln}, we have that
\begin{equation}\label{l_n(t)-l(t)}
\begin{aligned}
&\int_{\widetilde{T}_1}^T\left\|l_n(t)-l(t)\right\|_2^2ds\\
 &\hspace{8mm}\leq \left\|Gx(T)-m\left(Gx(T)\big|L^2(\mathscr{C}_{\widetilde{T}_1})\right)-Gx_n(T)+m\left(Gx_n(T)\big|L^2(\mathscr{C}_{\widetilde{T}_1})\right)\right\|_2^2\\
&\hspace{8mm}\leq2\left\|Gx(T)-Gx_n(T)\right\|_2^2+2\left\|m\left(Gx(T)\big|L^2(\mathscr{C}_{\widetilde{T}_1})\right)-m\left(Gx_n(T)\big|L^2(\mathscr{C}_{\widetilde{T}_1})\right)\right\|_2^2\\
&\hspace{8mm}\leq\mathcal{C}\|Gx(T)-Gx_n(T)\|_2^2.
\end{aligned}
\end{equation}
This, together with \eqref{G_nx_nT-Gx_nT}, implies that
\begin{equation}\label{the final of e(l_n(t)-l(t))}
\lim_{n\to\infty}\int_{\widetilde{T}_1}^T\left\|l_n(t)-l(t)\right\|_2^2ds=0.
\end{equation}
From \eqref{the representation of K}, \eqref{the estimate of K}, \eqref{the representation of Kn} and \eqref{the estimate of Kn}, we conclude that
\begin{align}\label{K_n(s,t)-K(s,t)}
  &\int_{\widetilde{T}_1}^T\left\|\int_t^T\left\{\mathcal{K}_n(s,t)-\mathcal{K}(s,t)\right\} ds\right\|^2_2dt\nonumber\\
&\hspace{8mm}\leq\left\{\left(\int_{\widetilde{T}_1}^T\left(\int_t^T\|\mathcal{K}_n(s,t)-\mathcal{K}(s,t)\|_2 ds\right)^2 dt\right)^{\frac{1}{2}}\right\}^2 \nonumber\\
&\hspace{8mm}\leq\left\{\int_{\widetilde{T}_1}^T\left(\int_t^T\|\mathcal{K}_n(s,t)-\mathcal{K}(s,t)\|^2_2 ds\right)^{\frac{1}{2}} dt\right\}^2\\
&\hspace{8mm}\leq3 \left\{ \int_{\widetilde{T}_1}^T\|A(t)^*\|_{\mathcal{L}(L^2(\mathscr{C}))}^2dt\left\{\sup_{t\in[\widetilde{T}_1,T]}\|y(t)-y_n(t)\|_2^2\right\}\right.\nonumber\\
  &\hspace{18.5mm}  +\int_{\widetilde{T}_1}^T\|C(t)^*\|_{\mathcal{L}(L^2(\mathscr{C}))}^2dt\int_{\widetilde{T}_1}^T\|Y(t)-Y_n(t)\|_{\mathcal{L}(L^2(\mathscr{C}))}^2dt  \nonumber\\
 &\hspace{18.5mm}   \left. +\int_{\widetilde{T}_1}^T\|M(t)\|^2_{\mathcal{L}(L^2(\mathscr{C}))}dt\left\{\sup_{t\in[\widetilde{T}_1,T]}\|x(t)-x_n(t)\|_2^2\right\}\right\}\nonumber.
 \end{align}
Combining with \eqref{yn-y+Y-Yn}-\eqref{K_n(s,t)-K(s,t)}, we obtain that
\begin{equation*}
\lim_{n\to\infty}\left\{\sup_{t\in[\widetilde{T}_1,T]}\|y(t)-y_n(t)\|_2^2+\int_{\widetilde{T}_1}^{T}\|Y(t)-Y_n(t)\|_2^2dt\right\}=0.
\end{equation*}
By repeating the above argument, we obtain the second and third equality in \eqref{finite approximation}.
\end{proof}

\noindent{Penghui Wang, School of Mathematics, Shandong University, Jinan 250100, Shandong, P. R. China, Email:
phwang@sdu.edu.cn}

\noindent{Shan Wang, School of Mathematics, Shandong University, Jinan 250100, Shandong, P. R. China, Email:
 202020244@mail.sdu.edu.cn
}

 \noindent{Shengkai Zhao, School of Mathematics, Shandong University, Jinan 250100, Shandong, P. R. China, Email:
202420320@mail.sdu.edu.cn}

\end{document}